\documentclass[10pt]{article}

\usepackage{amsthm}
\usepackage{mathtools}
\usepackage{amsmath}
\usepackage{bbm}
\usepackage{amsfonts}
\usepackage{amssymb}
\usepackage{xpatch}
\usepackage{enumitem}
\usepackage{graphicx}
\usepackage{thmtools,thm-restate}
\usepackage{multirow}
\usepackage{wrapfig}
\usepackage{tikz}
\tikzstyle{vertex}=[circle, draw, inner sep=0pt, minimum size=6pt]
\usetikzlibrary{arrows}

\usepackage{makecell}
\usepackage{footnote}
\usepackage{caption}
\makesavenoteenv{tabular}

\newcommand{\rr}[0]{\bar{r}}
\newcommand{\ccvar}[0]{c}
\newcommand{\lmin}[0]{\bar{\lambda}}

\newtheorem{lemma}{Lemma}

\newtheorem{definition}{Definition}
\newtheorem{assumption}{Assumption}

\newcommand{\R}{\mathbb{R}}
\newcommand{\E}{\mathbb{E}}

\newcommand{\defeq}{\coloneqq}

\newcommand{\del}{\partial}

\DeclarePairedDelimiter{\abs}{\lvert}{\rvert} %
\DeclarePairedDelimiter{\brk}{[}{]}
\DeclarePairedDelimiter{\crl}{\{}{\}}
\DeclarePairedDelimiter{\prn}{(}{)}
\DeclarePairedDelimiter{\nrm}{\|}{\|}
\DeclarePairedDelimiter{\norm}{\|}{\|}

\DeclarePairedDelimiter{\ceil}{\lceil}{\rceil}

\DeclareMathOperator*{\argmin}{arg\,min}

\newcommand{\mc}[1]{\mathcal{#1}}

\def\ddefloop#1{\ifx\ddefloop#1\else\ddef{#1}\expandafter\ddefloop\fi}
\def\ddef#1{\expandafter\def\csname 
bb#1\endcsname{\ensuremath{\mathbb{#1}}}}
\ddefloop ABCDEFGHIJKLMNOPQRSTUVWXYZ\ddefloop
\def\ddefloop#1{\ifx\ddefloop#1\else\ddef{#1}\expandafter\ddefloop\fi}
\def\ddef#1{\expandafter\def\csname 
b#1\endcsname{\ensuremath{\mathbf{#1}}}}
\ddefloop ABCDEFGHIJKLMNOPQRSTUVWXYZ\ddefloop
\def\ddef#1{\expandafter\def\csname 
c#1\endcsname{\ensuremath{\mathcal{#1}}}}
\ddefloop ABCDEFGHIJKLMNOPQRSTUVWXYZ\ddefloop
\def\ddef#1{\expandafter\def\csname 
h#1\endcsname{\ensuremath{\widehat{#1}}}}
\ddefloop ABCDEFGHIJKLMNOPQRSTUVWXYZ\ddefloop
\def\ddef#1{\expandafter\def\csname 
hc#1\endcsname{\ensuremath{\widehat{\mathcal{#1}}}}}
\ddefloop ABCDEFGHIJKLMNOPQRSTUVWXYZ\ddefloop
\def\ddef#1{\expandafter\def\csname 
t#1\endcsname{\ensuremath{\widetilde{#1}}}}
\ddefloop ABCDEFGHIJKLMNOPQRSTUVWXYZ\ddefloop
\def\ddef#1{\expandafter\def\csname 
tc#1\endcsname{\ensuremath{\widetilde{\mathcal{#1}}}}}
\ddefloop ABCDEFGHIJKLMNOPQRSTUVWXYZ\ddefloop

\newcommand{\indicator}[1]{\mathbbm{1}_{\crl{#1}}}    %Indicator
\usepackage{nicefrac}

\newsavebox\CBox

\newcommand{\ossgd}{\textsc{Regularized-Quadratic-Solver}}
\newcommand{\rqga}{\textsc{Regularized-Quadratic-Gradient-Access}}

\newcommand{\cqs}{\textsc{Constrained-Quadratic-Solver}}
\newcommand{\fedsn}{\textsc{FedSN}}
\newcommand{\fedac}{\textsc{FedAc}}
\newcommand{\fedsnlite}{\textsc{FedSN-Lite}}

\newcommand{\pars}[1]{\left( #1 \right)}

\newcommand{\inner}[2]{\left\langle #1,\, #2 \right\rangle}

\newcommand{\removed}[1]{}

\newcommand\restrict[1]{\raisebox{-.5ex}{$\big|$}_{#1}}

\usepackage{booktabs}
\usepackage{subcaption}
\DeclareSymbolFont{bbold}{U}{bbold}{m}{n}
\DeclareSymbolFontAlphabet{\mathbbold}{bbold}
\usepackage{thm-restate}
\usepackage{algorithm}
\usepackage[noend]{algorithmic}

\usepackage[round]{natbib}
\usepackage[left=1in,right=1in,top=1in,bottom=1in]{geometry}
\usepackage[colorlinks,allcolors=blue]{hyperref}
\usepackage{cleveref}

\newcommand{\updatet}{\Delta \tilde{x}}

\newcommand{\bu}[0]{\bar{u}}
\newcommand{\bg}[0]{\bar{\gamma}}
\newcommand{\boldI}[0]{\mathbf{I}}

\renewcommand{\algorithmicrequire}{\textbf{Input:}}
\renewcommand{\algorithmicensure}{\textbf{Output:}}

\title{\rule{\linewidth}{1.5pt}\\ \textbf{A Stochastic Newton Algorithm for Distributed Convex Optimization} \\\rule[8pt]{\linewidth}{1pt}}
\author{\normalsize
\begin{minipage}{0.3\textwidth}
\centering
\textbf{Brian Bullins} \\
Toyota Technological \\
Institute at Chicago \\
\small\url{bbullins@ttic.edu}
\end{minipage}
\begin{minipage}{0.3\textwidth}
\centering
\textbf{Kumar Kshitij Patel}\\
Toyota Technological\\ Institute at Chicago \\
\small\url{kkpatel@ttic.edu}
\end{minipage}
\begin{minipage}{0.3\textwidth}
\centering
\textbf{Ohad Shamir} \\
Weizmann Institute 
of Science \\
\small\url{ohad.shamir@weizmann.ac.il} \\ $ $
\end{minipage}\\\\
\normalsize
\begin{minipage}{0.3\textwidth}
\centering
\textbf{Nathan Srebro} \\
Toyota Technological \\
Institute at Chicago \\
\small\url{nati@ttic.edu}
\end{minipage}
\begin{minipage}{0.3\textwidth}
\centering
\textbf{Blake Woodworth}\\ 
Toyota Technological\\
Institute at Chicago \\
\small\url{blake@ttic.edu}
\end{minipage}

}
\date{}

\renewcommand{\P}{\mathbb{P}}
\begin{document}
\maketitle

\begin{abstract}
We propose and analyze a stochastic Newton algorithm for homogeneous distributed stochastic convex optimization, where each machine can calculate stochastic gradients of the same population objective, as well as stochastic {\em Hessian-vector products} (products of an independent unbiased estimator of the Hessian of the population objective with arbitrary vectors), with many such stochastic computations performed between rounds of communication.  We show that our method can reduce the number, and frequency, of required communication rounds compared to existing methods without hurting performance, by proving convergence guarantees for quasi-self-concordant objectives (e.g., logistic regression), alongside empirical evidence.
\end{abstract}
% \vspace{-1.2em}
\section{Introduction}
Stochastic optimization methods that leverage parallelism have proven immensely useful in modern optimization problems. Recent advances in machine learning have highlighted their importance as these techniques now rely on millions of parameters and increasingly large training sets.

While there are many possible ways of parallelizing optimization algorithms, we consider the intermittent communication setting \citep{zinkevich2010parallelized,cotter2011better,dekel2012optimal,shamir2014communication,woodworth2018graph,woodworth2021min}, where $M$ parallel machines work together to optimize an objective during $R$ rounds of communication, and where during each round each machine may perform some basic operation (e.g., access the objective by invoking some oracle) $K$ times, and then communicate with all other machines. An important example of this setting is when this basic operation gives independent, unbiased stochastic estimates of the gradient, in which case this setting includes algorithms like Local SGD \citep{zinkevich2010parallelized,Coppola2015:IPM,Zhou2018:Kaveraging,stich2019local,woodworth2020local}, Minibatch SGD \citep{dekel2012optimal}, Minibatch AC-SA \citep{ghadimi2012optimal}, and many others.

We are motivated by the observation of \cite{woodworth2020local} that for \emph{quadratic} objectives, first-order methods such as one-shot averaging \citep{zinkevich2010parallelized,zhang2013communication}---a special case of Local SGD with a single round of communication---can optimize the objective to a very high degree of accuracy. This prompts trying to reduce the task of optimizing general convex objectives to a short sequence of quadratic problems. Indeed, this is precisely the idea behind many second-order algorithms including Newton's method \citep{nesterov1994interior}, trust-region methods \citep{nocedal2006numerical}, and cubic regularization \citep{nesterov2006cubic}, as well as methods that go beyond second-order information \citep{nesterov2019implementable,bullins2020highly}.

Computing each Newton step requires solving, for convex $F$, a linear system of the form $\nabla^2 F(x) \Delta x = -\nabla F(x)$. This can, of course, be done directly using the full Hessian of the objective. However, this is computationally prohibitive in high dimensions, particularly when compared to cheap stochastic gradients. To avoid this issue, we reformulate the Newton step as the solution to a convex quadratic problem, 
$\min_{\Delta x} \frac{1}{2}\Delta x^\top\nabla^2 F(x) \Delta x\ +~\nabla F(x)^\top \Delta x$, which we then solve using one-shot averaging. Conveniently, computing stochastic gradient estimates for this quadratic objective does not require computing the full Hessian matrix, as it only requires stochastic gradients and \emph{stochastic Hessian-vector products}. This is attractive computationally since, for many problems, the cost of computing stochastic Hessian-vector products is similar to the cost of computing stochastic gradients, and both involve similar operations \citep{pearlmutter1994fast}. 
Furthermore, highlighting the importance of these estimates, recent works have relied on Hessian-vector products to attain faster rates for reaching approximate stationary points in both deterministic \citep{agarwal2017finding, carmon2018accelerated} and stochastic \citep{allen2018natasha2, arjevani2020second} non-convex optimization. 

\begin{table}\label{tab:results}
\caption{Convergence guarantees for different algorithms in the intermittent communication setting. Notation is as follows: $H$: smoothness; $U$: third-order-smoothness; $\sigma$: stochastic gradient variance; $\rho$: stochastic Hessian-vector product variance; $g(x;z)$: stochastic gradient oracle; $h(x,u;z')$: stochastic Hessian-vector product oracle (see \Cref{sec:setting} for complete details). For the sake of clarity, we omit additional constants and logarithmic factors.\\[-.5em]}
\centering
    \def\arraystretch{2}
	\begin{tabular}{ l l c}
	\toprule
		\makecell{Algorithm\\[0.3em](Reference)} & Convergence Rate & \makecell{Assumption,\\[0.3em]Oracle Access} \\
		\midrule
		\makecell{Local SGD\\[0.3em]\citep{woodworth2020local}} & $\frac{HB^2}{KR} + \frac{\sigma B}{\sqrt{MKR}} + \frac{H^{1/3}\sigma^{2/3}B^{4/3}}{K^{1/3}R^{2/3}}$ & \makecell{A1\\[0.3em]$g(x;z)$} \\
		\midrule
		\makecell{\fedac{}\\[0.3em]\citep{yuan2020federated}}  & $\frac{HB^2}{KR^2} + \frac{\sigma B}{\sqrt{MKR}} + \min\crl*{\frac{H^{1/3}\sigma^{2/3}B^{4/3}}{K^{1/3}R},\frac{H^{1/2}\sigma^{1/2}B^{3/2}}{K^{1/4}R}}$  & \makecell{A1\\[0.3em]$g(x;z)$}\\
		\midrule
		\makecell{Local SGD\\[0.3em]\citep{yuan2020federated}}  & $\frac{HB^2}{KR} + \frac{\sigma B}{\sqrt{MKR}} + \frac{U^{1/3}\sigma^{2/3}B^{5/3}}{K^{1/3}R^{2/3}}$  & \makecell{A3 ($3^{rd}$-order Smooth)\\[0.3em]$g(x;z)$}\\
		\midrule
		\makecell{\fedac{}\\[0.3em] \citep{yuan2020federated}} & $\frac{HB^2}{KR^2} + \frac{\sigma B}{\sqrt{MKR}} + \frac{H^{1/3}\sigma^{2/3}B^{4/3}}{M^{1/3}K^{1/3}R} + \frac{U^{1/3}\sigma^{2/3}B^{5/3}}{K^{1/3}R^{4/3}}$ & \makecell{A3 ($3^{rd}$-order Smooth)\\[0.5em]$g(x;z)$}\\
		\midrule
		\makecell{\fedsn{}\\[0.3em]\textbf{(\Cref{thm:combined-alg})}} & exp. decay + $\frac{HB^2}{KR}  + \frac{\sigma B}{\sqrt{MK}} + \frac{\rho B^2}{\sqrt{K}R}$  & \makecell{A2 (QSC)\\[0.3em]$g(x;z)$, $h(x,u;z')$}\\
		\toprule
	\end{tabular}
\end{table}

In the context of distributed optimization, second-order methods have shown promise in the empirical risk minimization (ERM) setting, whereby estimates of $F$ are constructed by distributing the component functions of the finite-sum problem across machines. Such methods which leverage this structure have since been shown to lead to improved communication efficiency \citep{shamir2014communication,zhang2015disco,reddi2016aide,wang2018giant,crane2019dingo,islamov2021distributed,gupta2021localnewton}. An important difference, however, is that these methods work in a batch setting, meaning they allow for repeated access to the same $K$ examples each round on each machine, giving a total of $MK$ samples. In contrast, we work in the stochastic (one-pass, streaming) setting, and so our model independently samples a fresh set of $MK$ examples per round, for a total of $MKR$ examples.

\nopagebreak[4]
\vspace{-5mm}
\subsubsection*{Our results} 
Our primary algorithmic contribution, which we sketch below and present in full in \Cref{sec:results}, is the method \textsc{Federated-Stochastic-Newton} (\fedsn{}), a distributed approximate Newton method which leverages the benefits of one-shot averaging for quadratic problems. We provide in \Cref{sec:results}, under the condition of quasi-self-concordance \citep{bach2010self}, the main guarantees of our method (\Cref{thm:combined-alg}). In \Cref{sec:compare} we show how, for some parameter regimes in terms of $M$, $K$, and $R$, our method may improve upon the rates of previous first-order methods, including Local SGD and \fedac{} \citep{yuan2020federated}.
In \Cref{sec:experiments}, we present a more practical version of our method, \fedsnlite{} (\Cref{alg:stochastic-newton-lite}), alongside empirical results which show that we can significantly reduce communication as compared to other first-order methods.

\begin{algorithm}[H]
\renewcommand\thealgorithm{}
   \caption{\hspace{-0.5em}\textbf{(Sketch)} \textsc{Federated-Stochastic-Newton}, a.k.a., \fedsn{}$(x_0)$}
   \label{alg:stochastic-newton-sketch}
\begin{algorithmic}
\STATE \hspace{-1em}(Operating on objective $F(\cdot)$ with stochastic gradient $g(\cdot;\cdot)$ and Hessian-vector product $h(\cdot;\cdot,\cdot)$ oracles.)
    \REQUIRE $x_0 \in \R^d$.\renewcommand{\algorithmicrequire}{\textbf{Hyperparameters:}}
    \REQUIRE $T$: main iterations; $\bar{\xi}$: local stability parameter; and $N$: binary search iterations (see \Cref{tab:hyperparams}).
    \ENSURE Approximate solution to $\min_x F(x)$
   \FOR{$t = 0,1,\dots,T-1$}
   \STATE $\Lambda = \crl*{\lambda_1,\dots,\lambda_N}$ \hfill $\triangleright$ Set of regularization parameters
   \WHILE{$\Lambda \neq \varnothing$}
   \STATE{$\lambda = \textrm{Median}(\Lambda)$}
   \STATE Define: $Q_\lambda(u) = \frac{1}{2}u^\top (\bar{\xi}\nabla^2 F(x)+\lambda \boldI) u + \nabla F(x)^\top u$\hfill $\triangleright$ Regularized quadratic subproblem\vspace{0.6em}
\STATE Run SGD on $Q_{\lambda}(u)$ independently on each machine (starting with $u_0=0$) for $K$ steps, then communicate the final iterate $u_K^m$ from each machine $m$, and average them to obtain $\tilde{u}=\frac{1}{M}\sum_{m=1}^M u_K^m$.\vspace{0.6em}
    \IF{Sufficiently good $\lambda$ is found}
    \STATE $\updatet_t = \tilde{u}$
    \STATE \textbf{break}
    \ELSE
    \STATE Reduce the size of $\Lambda$ by half.
    \ENDIF
   \ENDWHILE
   \STATE Update: $x_{t+1} = x_t + \updatet_t$
   \ENDFOR
   \STATE \textbf{Return:} $x_T$
\end{algorithmic}
\end{algorithm}

% \vspace{-0.5em}
\section{Preliminaries}\label{sec:setting}
We consider the following optimization problem:
\begin{equation}
\min_{x\in\R^d} F(x),
\end{equation}
and throughout we use $F^*$ to denote the minimum of this problem. We further use $\nrm*{\cdot}$ to denote the standard $\ell_2$ norm, we let $\nrm*{x}_{\mathbf{A}} \defeq \sqrt{x^\top \mathbf{A} x}$ for a positive semidefinite matrix $\mathbf{A}$, and we let $\boldI$ denote the identity matrix of order $d$.

Next, we establish several sets of assumptions, beginning with those which are standard for smooth, stochastic, distributed convex optimization. We would note that we are working in the \emph{homogeneous} distributed setting (i.e., each machine may access the same distribution), rather than the heterogeneous setting \citep{khaled2019first, karimireddy2019scaffold, koloskova2020unified, woodworth2020minibatch, khaled2020tighter}.
\begin{assumption}[A1]\label{assump:a1}\
\begin{enumerate}[label=(\alph*), itemsep=0.1em]
    \item $F$ is convex, differentiable, and $H$-smooth, i.e., for all $x,y \in \R^d$, \[F(y) \leq F(x) + \nabla F(x)^\top(y-x) + \frac{H}{2}\norm{y-x}^2.\]
    \item There is a minimizer $x^* \in \argmin_x F(x)$ such that $\norm{x^*} \leq B$.
    \item We are given access to a stochastic first-order oracle in the form of an estimator $g : \R^d \times \mathcal{Z} \mapsto \R^d$, and a distribution $\mathcal{D}$ on $\mathcal{Z}$ such that, for any $x \in \R^d$ queried by the algorithm, the oracle draws $z \sim \mathcal{D}$, and the algorithm observes an estimate $g(x;z)$ that satisfies:
    \begin{enumerate}[label=(\roman*), itemsep=0.1em]
        \item $g(x;z)$ is an unbiased gradient estimate, i.e., $\E_z g(x;z) = \nabla F(x)$.
        \item $g(x;z)$ has bounded variance, i.e., $\E_z \norm{g(x;z) - \nabla F(x)}^2 \leq \sigma^2.$
    \end{enumerate}
\end{enumerate}
\end{assumption}

In order to provide guarantees for Newton-type methods, we will require additional notions of smoothness. In particular, we consider $\alpha$-quasi-self-concordance (QSC) \citep{bach2010self}, which for convex and three-times differentiable $F$ is satisfied for $\alpha \geq 0$ when, for all $x \in \textrm{dom}(F)$, $v, u \in \R^d$,
\begin{equation*}
\abs{\nabla^3 F(x)[v,u,u]} \leq \alpha\norm{v}\pars{\nabla^2 F(x)[u,u]},
\end{equation*}
where we define
\begin{equation*}
    \nabla^k F(x)[u_1,u_2,\dots,u_k] \defeq \frac{\del^k}{\del u_1, \del u_2, \dots, \del u_k}\restrict{t_1=0,t_2=0,\dots,t_k=0} F(x+t_1 u_1 + t_2 u_2 + \dots + t_k u_k),
\end{equation*}
for $k \geq 1$, i.e., the $k^{th}$ directional derivative of $F$ at $x$ along the directions $u_1,u_2,\dots,u_k$.
Related to this is the condition of $\alpha$-self-concordance, which has proven useful for classic problems in linear optimization \citep{nesterov1994interior}, whereby for all $x \in \textrm{dom}(F), u \in \R^d$,
\begin{equation*}
    \abs{\nabla^3 F(x)[u,u,u]} \leq 2\alpha\pars{\nabla^2 F(x)[u,u]}^{3/2}.
\end{equation*}

Though quasi-self-concordance is perhaps not as widely studied as self-concordance, recent work has brought its usefulness to light in the context of machine learning \citep{bach2010self, karimireddy2018global, carmon2020acceleration}. Notably, for logistic regression, i.e., problems of the form
\begin{equation}\label{eq:logistic-regression}
    \min_x F(x) = \frac{1}{N}\sum\limits_{i = 1}^N \log\prn*{1+e^{-b_i\inner{a_i}{x}}},
\end{equation}
we observe that $\alpha$-quasi-self-concordance holds with $\alpha \leq \max_{i}\crl*{\norm{b_i a_i}}$. Interestingly, this function is \emph{not} self-concordant, thus highlighting the importance of introducing the notion of QSC for such problems, and indeed, neither of these conditions implies the other in general.\footnotemark \footnotetext{For the other direction, note that $F(x)=-\ln(x)$ is $1$-self-concordant but not quasi-self-concordant.}

We now introduce further assumptions in terms of both additional oracle access and other smoothness notions. The following outlines the requirements for the stochastic Hessian-vector products, though we again stress that the practical cost of such an oracle is often on the order of that for stochastic gradients \citep{pearlmutter1994fast,allen2018natasha2}.

\begin{assumption}[A2]\label{assump:a2}\ In addition to Assumption~\ref{assump:a1}, we have: \\[-1.5em]
\begin{enumerate}[label=(\alph*), itemsep=0.1em]
    \item $F$ is three-times differentiable and $\alpha$-quasi-self-concordant, i.e., for all $x,v,u \in \R^d,$
    \[\abs*{\nabla^3 F(x)[v,u,u]} \leq \alpha\nrm{v}\nabla^2 F(x)[u,u].\]
    \item We are given access to a stochastic Hessian-vector product oracle in the form of an estimator $h : \R^d \times \R^d \times \mathcal{Z} \mapsto \R^d$, and a distribution $\mathcal{D}$ on $\mathcal{Z}$ such that, for any pair $x,u \in \R^d$ queried by the algorithm, the oracle draws $z' \sim \mathcal{D}$, and the algorithm observes an estimate $h(x,u;z')$ that satisfies:
    \begin{enumerate}[label=(\roman*), itemsep=0.1em]
        \item $h(x,u;z')$ is an unbiased Hessian-vector product estimate, i.e., $\E_{z'} h(x,u;z') = \nabla F^2(x)u$.
        \item $h(x,u;z')$ has bounded variance of the form $\E_{z'} \norm{h(x,u;z') - \nabla^2 F(x)u}^2 \leq \rho^2\norm{u}^2.$
    \end{enumerate}
\end{enumerate}
\end{assumption}

Meanwhile, other works \cite[e.g.,][]{yuan2020federated} require different control over third-order smoothness and fourth central moment.  We do not require this assumption in our analysis, and include it here for comparison.

\begin{assumption}[A3]\label{assump:a3}\ In addition to Assumption~\ref{assump:a1}, we have: \\[-1.5em]
\begin{enumerate}[label=(\alph*), itemsep=0.1em]
    \item $F$ is twice-differentiable and $U$-third-order-smooth, i.e., for all $x,y \in \R^d$, 
        \[F(y) \leq F(x) + \nabla F(x)^\top(y-x) + \frac{1}{2}\inner{\nabla^2 F(x)(y-x)}{y-x} + \frac{U}{6}\norm{y-x}^3.\]
    \item $g(x;z)$ has bounded fourth central moment, i.e., $\E_z \norm{g(x;z) - \nabla F(x)}^4 \leq \sigma^4.$
\end{enumerate}
\end{assumption}

% \vspace{-0.8em}
\section{Main results}\label{sec:results}
We begin by describing our main algorithm, \fedsn{} (\Cref{alg:stochastic-newton}).
Namely, our aim is to solve convex minimization problems $\min_x F(x)$, subject to Assumption~\ref{assump:a2}.

\setcounter{algorithm}{0}
\begin{algorithm}[H]
   \caption{\textsc{Federated-Stochastic-Newton}, a.k.a., \fedsn{}$(x_0)$}
   \label{alg:stochastic-newton}
\begin{algorithmic}
\STATE \hspace{-1em}(Operating on objective $F(\cdot)$ with stochastic gradient $g(\cdot;\cdot)$ and Hessian-vector product $h(\cdot;\cdot,\cdot)$ oracles.)
    \REQUIRE $x_0 \in \R^d$.\renewcommand{\algorithmicrequire}{\textbf{Hyperparameters:}}
    \REQUIRE $T$: main iterations; and $\bar{\xi}$: local stability (see \Cref{tab:hyperparams}).
    \ENSURE Approximate solution to $\min_x F(x)$ \hfill $\triangleright$ See \Cref{thm:combined-alg}
   \FOR{$t = 0,1,\dots,T-1$}
   \STATE $\updatet_t = \textrm{\cqs{}}\prn*{x_t}$\hfill $\triangleright$ Approx. $\min_{u:\norm{u}\leq \frac{1}{2}\rr} \frac{\bar{\xi}}{2}u^\top \nabla^2 F(x_t) u + \nabla F(x_t)^\top u$
   \STATE Update: $x_{t+1} = x_t + \updatet_t$
   \ENDFOR
   \STATE \textbf{Return:} $x_T$
\end{algorithmic}
\end{algorithm}

We will rely throughout the paper on several hyperparameter settings and parameter functions, which we collect in \Cref{tab:hyperparams,tab:paramfuncs}. Recall that $M$ is the amount of parallel workers, $R$ is the number of rounds of communication, $K$ is the number of basic operations performed between rounds of communication, and $H$, $B$, $\sigma$, $\alpha$, and $\rho$ are as defined in Assumptions \ref{assump:a1} and \ref{assump:a2}.
\begin{table}[H]
\def\arraystretch{1.5}%
\centering
\begin{tabular}{ l l } 
 \toprule
 Hyperparameter Setting  & Description\\
 \midrule 
$T \defeq \Bigg\lfloor\frac{R}{4\zeta}\log^2\prn*{\prn*{\frac{R}{\zeta}}}\Bigg\rfloor$ (for $\zeta = 4096 + 4\prn*{80 + 32\log K + 24\log(1+2\alpha B)}^2$) & Main iterations\\[0.3em]
$\beta := 0$ & Momentum\\
 $\rr \defeq \min\crl*{\frac{32B}{T}\log(TK),\,\frac{1}{5\alpha}}$ & Trust-region radius\\
 $\bar{\xi} \defeq \exp(\alpha \rr)$ & Local stability\\
 $\lmin \defeq \max\crl*{\frac{2eH}{K-2},\,\frac{2\rho}{\sqrt{K}},\, \frac{32eH\log(51200)}{K},\, \frac{4\rho\sqrt{2\log(51200)}}{\sqrt{K}},\,\frac{320\sqrt{2}\rho}{\sqrt{MK}},\,\frac{320\sigma}{\rr\sqrt{MK}},\, \frac{8eH}{K-16}}$ & Regularization bound\\
 $N \defeq \left\lceil1 + \frac{5}{2}\log\frac{H(B + 5T\rr)}{3\lmin\rr}\right\rceil$ & Binary search iterations\\
 $C \defeq \left\lceil 8\log\prn*{\ceil{\log_2 N}\prn*{4 + \frac{e H}{\lmin} + \frac{80H(B + 5T\rr)}{\lmin \rr}}} \right\rceil$ & Reg.\ quadratic repetitions\\[0.8em]
 \toprule 
\end{tabular}
\caption{Hyperparameters $T$, $\beta$, $\rr$, $\bar{\xi}$, $\bar{\lambda}$, $N$, and $C$, as used by \fedsn{} and its subroutines.}
\label{tab:hyperparams}
\end{table}

\begin{table}[H]
\def\arraystretch{1}%
\centering
\begin{tabular}{ l l } 
 \toprule
 Parameter Function & Description\\
 \midrule \\
$\eta_k(\lambda) := \begin{cases}
\eta_\lambda & K \leq \frac{2}{\lambda}\max\crl*{\bar{\xi}H+\lambda,\, \frac{\rho^2}{\lambda}}\textrm{ or }k < \frac{K}{2} \\
\frac{4}{\lambda\prn*{\frac{8}{\lambda}\max\crl*{\bar{\xi}H+\lambda,\,\frac{\rho^2}{\lambda}} +k-\frac{K}{2}}} & K > \frac{2}{\lambda}\max\crl*{\bar{\xi}H+\lambda,\, \frac{\rho^2}{\lambda}}
\textrm{ and }k \geq \frac{K}{2}
\end{cases}$ & Reg.\ quadratic stepsizes \\
& \\ 
$w_k(\lambda) := \begin{cases}
(1-\lambda\eta_\lambda + \eta_\lambda^2\rho^2)^{-k-1} & K \leq \frac{2}{\lambda}\max\crl*{\bar{\xi}H+\lambda,\, \frac{\rho^2}{\lambda}} \\
0 & K > \frac{2}{\lambda}\max\crl*{\bar{\xi}H+\lambda,\, \frac{\rho^2}{\lambda}} \textrm{ and }k < \frac{K}{2} \\
\frac{8}{\lambda}\max\crl*{\bar{\xi}H+\lambda,\,\frac{\rho^2}{\lambda}}\\\qquad +\ k - \frac{K}{2} - 1  & K > \frac{2}{\lambda}\max\crl*{\bar{\xi}H+\lambda,\, \frac{\rho^2}{\lambda}} \textrm{ and }k \geq \frac{K}{2}
\end{cases}$ \rule{0pt}{1em} & Reg.\ quadratic weights\\
&\\
 \toprule 
\end{tabular}
\caption{Parameter functions $\eta_k(\lambda)$ and $w_k(\lambda)$, as used by \fedsn{} and its subroutines, where $\bar{\xi}$ is as defined in \Cref{tab:hyperparams}, and where $\eta_\lambda$ denotes $\eta(\lambda) := \frac{1}{2}\min\crl*{\frac{1}{\bar{\xi}H+\lambda},\ \frac{\lambda}{\rho^2}}$.}
\label{tab:paramfuncs}
\end{table}

% \vspace{-0.5em}

Among our assumptions, we note in particular the condition of quasi-self-concordance, under which several works have provided efficient optimization methods. For example, \cite{bach2010self} analyzes Newton's method under QSC conditions, in a manner analogous to that of standard self-concordance analyses, to establish its behavior in the region of quadratic $(\log\log(1/\epsilon))$ convergence. More recently, both \cite{karimireddy2018global} and  \cite{carmon2020acceleration} have presented methods which rely instead on a trust-region approach, whereby, for a given iterate $x_t$, each iteration amounts to approximately solving a \emph{constrained} subproblem of the form
\begin{equation*}
\min_{\Delta x: \nrm*{\Delta x}\leq c} \frac{\xi}{2}\Delta x^\top \nabla^2 F(x_t) \Delta x + \nabla F(x_t)^\top \Delta x,
\end{equation*}
for some $\xi \geq 1$ and problem-dependent radius $c > 0$. This stands in constrast to the \emph{unconstrained} minimization problem $\min_{\Delta x} \frac{\xi}{2}\Delta x^\top\nabla^2 F(x) \Delta x + \nabla F(x)^\top \Delta x$, which, as we may recall, forms the basis of the standard (damped) Newton method. \cite{carmon2020acceleration} further use their trust-region subroutine to approximately implement a certain $\ell_2$-ball minimization oracle, which they combine with an acceleration scheme \citep{monteiro2013accelerated}.

These results show, at a high level, that as long as the radius of the constrained quadratic (trust-region) subproblem is not too large, it is possible to make sufficient progress on the \emph{global} problem by approximately solving the quadratic subproblem.
Our method proceeds in a similar fashion: each iteration of \Cref{alg:stochastic-newton} provides an approximate solution to a constrained quadratic problem. To begin, we follow \cite{karimireddy2018global} in defining $\delta(r)$-local (Hessian) stability.

\begin{definition}\label{def:localstability}
Let $\delta: \R^+ \mapsto \R^+$. We say that a twice-differentiable and convex function $F$ is $\delta(r)$-locally stable if, for any $r > 0$ and any $x,y \in \R^d$ $(x \neq y)$ such that $\nrm{x - y} \leq r$ and $\nrm{x - y}_{\nabla^2 F(x)}>0$,
\[
\nrm{x - y}_{\nabla^2 F(y)} \leq \delta(r) \nrm{x - y}_{\nabla^2 F(x)}.
\] 
\end{definition}
As the next lemma shows, quasi-self-concordance is sufficient to provide this type of local stability.
\begin{lemma}[Theorem I \citep{karimireddy2018global}]\label{lem:qsc-sensitivity}
If $F$ is $\alpha$-quasi-self-concordant, then $F$ is $\delta(r) = \exp(\alpha r)$-locally stable.
\end{lemma}
The advantage of local stability is that it ensures that approximate solutions to locally-defined constrained quadratic problems can guarantee progress on the global objective, and we state this more formally in the following lemma. Note that this lemma is similar to \cite[Theorem IV][]{karimireddy2018global}, though we allow for an additive error in the subproblem solves in addition to the multiplicative error, and its proof can be found in \Cref{app:stochastic-newton}.
\begin{restatable}{lemma}{newtonanalysis}\label{lem:qsc-newton}
Let $F$ satisfy Assumption \ref{assump:a2} and be $\delta(r)$-locally stable for $\delta : \R^+ \mapsto \R^+$, let $x_0 \in \R^d$ be as input to \fedsn{} (\Cref{alg:stochastic-newton}), let $\ccvar > 0$, let $\theta \in [0,1)$, and define 
\begin{equation*}
    Q_t(\Delta x) \defeq \frac{\delta\prn*{5\ccvar}}{2} \Delta x^\top \nabla^2 F(x_t)\Delta x + \nabla F(x_t)^\top \Delta x,
\end{equation*} where $x_t$ is the $t^{th}$ iterate of \Cref{alg:stochastic-newton}. Furthermore, suppose we are given that in each iteration of \Cref{alg:stochastic-newton}, $\nrm{\updatet_t} \leq 5\ccvar$ and
\[
\E Q_t(\updatet_t) - \min_{\Delta x:\nrm{\Delta x}\leq\frac{1}{2}\ccvar} Q_t(\Delta x) \leq \theta\prn*{Q_t(0) - \min_{\Delta x:\nrm{\Delta x}\leq\frac{1}{2}\ccvar} Q_t(\Delta x)} + \epsilon,
\]
 for $\epsilon > 0$. Then for each $T \geq 0$, \Cref{alg:stochastic-newton} guarantees
\[
\E F(x_T) - F^* \leq \E\brk*{F(x_0) - F^*}\exp\prn*{-\frac{T\ccvar(1-\theta)}{2B\delta(\frac{1}{2}\ccvar)\delta\prn*{5\ccvar}}} + \frac{2B\delta(\frac{1}{2}\ccvar)d\prn*{5\ccvar}\epsilon}{\ccvar(1-\theta)}\ .
\]
\end{restatable}
We have now seen how to turn approximate solutions of constrained quadratic problems into an approximate minimizer of the overall objective. We next need to ensure that the output of our method \cqs{} (\Cref{alg:constrained-quadratic}) meets the conditions of \Cref{lem:qsc-newton}. As previously discussed, \citet{woodworth2020local} showed that first-order methods can very accurately optimize \emph{unconstrained} quadratic objectives using a single round of communication; however, here we need to optimize a quadratic problem subject to a norm constraint. Our constrained quadratic solver is thus based on the following idea: the minimizer of the \emph{constrained} problem $\min_{x:\nrm{x} \leq c} Q(x)$ is the same as the minimizer of the \emph{unconstrained} problem $\min_{x} Q(x) + \frac{\lambda^*}{2}\nrm{x}^2$ for some problem-dependent regularization parameter $\lambda^*$. While the algorithm does not know  what $\lambda^*$ should be a priori, we show that it can be found with sufficient confidence using binary search. \Cref{lem:constrained-quadratic-correctness}, proven in \Cref{app:constrained-quadratic}, provides the relevant guarantees.

\begin{algorithm}[H]
\caption{\cqs{}$(x)$}
\label{alg:constrained-quadratic}
\begin{algorithmic}
\STATE \hspace{-1em}(Operating on objective $F(\cdot)$ with stochastic gradient $g(\cdot;\cdot)$ and Hessian-vector product $h(\cdot;\cdot,\cdot)$ oracles.)
\REQUIRE $x \in \R^d$. \renewcommand{\algorithmicrequire}{\textbf{Hyperparameters:}}
\REQUIRE $\rr$: trust-region radius; $\bar{\xi}$: local stability; $\bar{\lambda}$: regularization bound; $N$: binary search iterations; and $C$: reg. quadratic repetitions (see \Cref{tab:hyperparams}).
\ENSURE Approximate solution to $\min_{u:\norm{u}\leq \frac{1}{2}\rr} \frac{\bar{\xi}}{2}u^\top \nabla^2 F(x) u + \nabla F(x)^\top u$ \hfill $\triangleright$ See \Cref{lem:constrained-quadratic-correctness}
\STATE $\Lambda_1 = \crl*{\bar{\lambda}\prn*{\frac{3}{2}}^{n-1}\,:\, n = 1,\dots,N}$
\STATE $i \gets 1$
\WHILE{$\Lambda_i \neq \varnothing$}
\STATE $\lambda^{(i)} = \textrm{Median}(\Lambda_i)$
\FOR{$c = 1,\dots,C$}
\STATE $\tilde{u}^{(i,c)} = \textrm{\ossgd{}}(x, \lambda^{(i)})$
\ENDFOR
\IF{$\abs*{\crl*{\tilde{u}^{(i,c)} : \nrm*{\tilde{u}^{(i,c)}}\in \brk*{\frac{3}{2}\rr,\,\frac{7}{2}\rr}}} > \frac{C}{2}$}
\STATE $\tilde{u} = \textrm{\ossgd{}}(x,\lambda^{(i)})$
\STATE \label{algline:lambda-accepted}\textbf{Return:} $\hat{u} = \min\crl*{1,\,\frac{5\rr}{\nrm{\tilde{u}}}}\tilde{u}$
\ELSIF{$\abs*{\crl*{\tilde{u}^{(i,c)} : \nrm*{\tilde{u}^{(i,c)}} \leq \frac{5}{2}\rr}} > \frac{C}{2}$}
\STATE \label{algline:lower} $\Lambda_{i+1} = \crl*{\lambda' \in \Lambda_i\,:\, \lambda' < \lambda^{(i)}}$
\ELSIF{$\abs*{\crl*{\tilde{u}^{(i,c)} : \nrm*{\tilde{u}^{(i,c)}} > \frac{5}{2}\rr}} > \frac{C}{2}$}
\STATE \label{algline:higher} $\Lambda_{i+1} = \crl*{\lambda' \in \Lambda_i\,:\, \lambda' > \lambda^{(i)}}$
\ELSE
\STATE \label{algline:return-0} \textbf{Return:} $\hat{u} = 0$
\ENDIF
\STATE $i \gets i+1$
\ENDWHILE
\STATE $\tilde{u} = \textrm{\ossgd{}}(x, \lmin)$
\STATE \label{algline:use-lmin} \textbf{Return:} $\hat{u} = \min\crl*{1,\,\frac{5\rr}{\nrm{\tilde{u}}}}\tilde{u}$
\end{algorithmic}
\end{algorithm}

\begin{restatable}{lemma}{cqscorrectness}\label{lem:constrained-quadratic-correctness}
Let $F$ satisfy Assumption~\ref{assump:a2}, let $x$ be as input to \cqs{} (\Cref{alg:constrained-quadratic}), let $\bar{\xi}$ be as in \Cref{tab:hyperparams}, define 
\begin{equation*}
    Q(u) \defeq \frac{\bar{\xi}}{2}u^\top \nabla^2 F(x) u + \nabla F(x)^\top u,\qquad Q_\lambda(u) \defeq \frac{1}{2}u^\top (\bar{\xi}\nabla^2 F(x) + \lambda \boldI) u + \nabla F(x)^\top u,
\end{equation*}
\begin{equation*}
    u_\lambda^* \defeq \arg\min_u Q_\lambda(u),\qquad r^*(\lambda) \defeq \nrm{u_\lambda^*},
\end{equation*}
and let $\lambda_r$ denote, for any $r>0$, the value such that $r^*(\lambda_r) = r$. Let $\hat{u}$ be the output of \Cref{alg:constrained-quadratic} for hyperparameters $\rr$, $\bar{\xi}$, $\bar{\lambda}$, $N$ and $C$ as in \Cref{tab:hyperparams}, and suppose the output $\tilde{u}_\lambda$ of \ossgd{}$(x,\lambda)$ satisfies for all $\lambda \geq \lmin$ that
\[
\E Q_\lambda(\tilde{u}_\lambda) - \min_u Q_\lambda(u) \leq \epsilon(\lambda) \defeq \frac{\lambda(r^*(\lambda)^2 + \rr^2)}{800}\ .
\]
Then $\nrm{\hat{u}} \leq 5\rr$ and
\[
\E Q(\hat{u}) - \min_{u:\nrm{u}\leq\frac{1}{2}\rr} Q(u) \leq \frac{3}{4}\prn*{Q(0) - \min_{u:\nrm{u}\leq\frac{1}{2}\rr} Q(u)} + \epsilon(\lambda_{4\rr}) + \frac{\lmin \rr^2}{4}\ .
\]
\end{restatable}

We now show that using one-shot averaging \citep{zinkevich2010parallelized, zhang2012communication} with $M$ machines---i.e.,~averaging the results of $M$ independent runs of SGD---suffices to solve each quadratic problem to the desired accuracy. 
The following lemma, which we prove in \Cref{app:proof-sgd-with-multiplicative-noise}, establishes that \ossgd{} (\Cref{alg:one-shot-averaging}) supplies \Cref{alg:constrained-quadratic} with an output $\hat{u}$ that satisfies the conditions of \Cref{lem:constrained-quadratic-correctness}. 
\begin{restatable}{lemma}{sgdwithmultnoise}\label{lem:sgd-with-multiplicative-noise}
Let $F$ satisfy Assumption~\ref{assump:a2}, let $x \in \R^d$, $\lambda \in \R^+$ be as input to \ossgd{} (\Cref{alg:one-shot-averaging}), let $Q_\lambda(u) = \frac{1}{2}u^\top (\bar{\xi}\nabla^2 F(x)+\lambda \boldI) u + \nabla F(x)^\top u$, let $Q_\lambda^*\defeq \min_u Q_\lambda(u)$, let $u^* \defeq \argmin_u Q_\lambda(u)$, and let stochastic first-order and stochastic Hessian-vector product oracles for $F$, as defined in Assumptions \ref{assump:a1} and \ref{assump:a2}, respectively, be available for each call to \rqga{} (\Cref{alg:quad-grad-access}), for either Case 1 \emph{(\texttt{Different-Samples})} or Case 2 \emph{(\texttt{Same-Sample})}.
Let $\hat{u}$, as output by \Cref{alg:one-shot-averaging}, be a weighted average of the iterates of $M$ independent runs of SGD with stepsizes $\eta_0(\lambda),\dots,\eta_{K-1}(\lambda)$, i.e.,
\[
\hat{u} = \frac{1}{M\sum_{k=0}^{K-1}w_k}\sum_{m=1}^M\sum_{k={}0}^{K-1}w_k u_k^m.
\]
Then, for both Cases 1 and 2, 
\[
\E Q_\lambda(\hat{u}) - Q_\lambda^* \leq 
\begin{cases}
2\max\crl*{\bar{\xi}H+\lambda, \frac{\rho^2}{\lambda}}\nrm*{u^*}^2\min\crl*{\frac{1}{K},\,\exp\prn*{-\frac{K+1}{4}\min\crl*{\frac{\lambda}{\bar{\xi}H+\lambda},\ \frac{\lambda^2}{\rho^2}}}} + \frac{2(\sigma^2 + \rho^2\nrm*{u^*}^2)}{\lambda MK} \\[0.5em] \hfill \textrm{if } K \leq \frac{2}{\lambda}\max\crl*{\bar{\xi}H+\lambda,\, \frac{\rho^2}{\lambda}} \\
\\
96\lambda\nrm*{u^*}^2\exp\prn*{ -\frac{K}{8}\min\crl*{\frac{\lambda}{\bar{\xi}H+\lambda},\ \frac{\lambda^2}{\rho^2}}} + \frac{96(\sigma^2 + \rho^2\nrm*{u^*}^2)}{\lambda MK}
\hfill\textrm{ if } K > \frac{2}{\lambda}\max\crl*{\bar{\xi}H+\lambda,\, \frac{\rho^2}{\lambda}}.
\end{cases}
\]
\end{restatable}
\begin{algorithm}[H]
\caption{\ossgd{}$(x,\lambda)$}
\label{alg:one-shot-averaging}
\begin{algorithmic}
\STATE \hspace{-1em}(Operating on objective $F(\cdot)$ with stochastic gradient $g(\cdot;\cdot)$ and Hessian-vector product $h(\cdot;\cdot,\cdot)$ oracles.)
\REQUIRE $x\in \R^d$, $\lambda \in \R^+$.\\
\renewcommand{\algorithmicrequire}{\textbf{Hyperparameters:}}
\REQUIRE $\beta$: momentum; $\bar{\xi}$: local stability; and parameter functions $\eta_k(\lambda)$, $w_k(\lambda)$ (see \Cref{tab:hyperparams,tab:paramfuncs}).
\ENSURE Approximate solution to $\min_u Q_\lambda(u) = \frac{1}{2}u^\top (\bar{\xi}\nabla^2 F(x)+\lambda \boldI) u + \nabla F(x)^\top u$\hfill $\triangleright$ See \Cref{lem:sgd-with-multiplicative-noise}
\STATE Initialize: $u_0^1,\dots,u_0^M = 0$ \hfill $\triangleright$ Initial iterates on each machine
\FOR{Each machine $m=1,\dots,M$ in parallel}
\FOR{$k = 0,\dots,K-1$}
\STATE $\gamma(u_{k}^m;z_k^m,z_k^{'m}) =$ \rqga{}$(x,u_k^m,\lambda)$
\STATE $u_{k+1}^m = u_k^m - \eta_k(\lambda) \gamma(u_{k}^m;z_k^m,z_k^{'m}) + \indicator{k>0}\beta(u_k^m - u_{k-1}^m)$\ \footnotemark
\ENDFOR
\ENDFOR
\STATE \textbf{Return:} $\tilde{u} = \frac{1}{M\sum_{k=1}^K w_k(\lambda)}\sum_{m=1}^M\sum_{k=1}^K w_k(\lambda) u_k^m$
\end{algorithmic}
\end{algorithm}
\footnotetext{We add heavy-ball/Polyak momentum in this step with momentum parameter $\beta$. Our theoretical results do not require any momentum, and thus \fedsn{} is analyzed for $\beta=0$. For our experiments we compare the algorithms both with and without momentum, i.e., $\beta=0$ or optimally tuned $\beta\in\{0.1, 0.3, 0.5, 0.7, 0.9\}$.}
\begin{algorithm}[H]
\caption{\rqga{}$(x,u,\lambda)$}
\label{alg:quad-grad-access}
\begin{algorithmic}
\STATE \hspace{-1em}(Operating on objective $F(\cdot)$ with stochastic gradient $g(\cdot;\cdot)$ and Hessian-vector product $h(\cdot;\cdot,\cdot)$ oracles.)
\REQUIRE $x, u\in \R^d$, $\lambda \in \R^+$.\\
\renewcommand{\algorithmicrequire}{\textbf{Hyperparameters:}}
\REQUIRE $\bar{\xi}$: local stability (see \Cref{tab:hyperparams}).
\ENSURE $\gamma(u;z,z')$ s.t. $\E_{z,z'}[\gamma(u;z,z')] = \nabla Q_\lambda(u)$ and $\E_{z,z'}\nrm*{\gamma(u;z,z')-\nabla Q_\lambda(u)}^2 \leq \sigma^2 + \rho^2\nrm*{u}^2$\vspace{0.5em}
\STATE \textbf{Case 1:} \texttt{Different-Samples} ($z, z'$ drawn independently for each stochastic oracle)\vspace{-0.8em}\begin{itemize}
    \item Query the stochastic first-order oracle at $x$ (as in Assumption~\ref{assump:a1}(c)), so that the oracle draws $z\sim\mathcal{D}$, and observe $g(x;z)$\vspace{-0.8em}
    \item Query the stochastic Hessian-vector product oracle at $x$ and $u$ (as in Assumption~\ref{assump:a2}(b)), so that the oracle draws $z' \sim\mathcal{D}$, and observe $h(x,u;z')$
\end{itemize}
\STATE \textbf{Case 2:} \texttt{Same-Sample} (Same $z'=z$ used for both stochastic oracles)\vspace{-0.8em}\begin{itemize}
\item Query the stochastic first-order oracle at $x$ (as in Assumption~\ref{assump:a1}(c)), so that the oracle draws $z\sim\mathcal{D}$, and observe $g(x;z)$\vspace{-0.8em}
\item Query the stochastic Hessian-vector product oracle at $x$ and $u$ (as in Assumption~\ref{assump:a2}(b)) for $z'=z$, and observe $h(x,u;z')$
\end{itemize}
\STATE $\gamma(u;z,z') \defeq \bar{\xi} h(x,u;z') + \lambda u + g(x;z)$\vspace{0.5em}
\STATE \textbf{Return:} $\gamma(u;z,z')$
\end{algorithmic}
\end{algorithm}
Our analysis for \Cref{alg:one-shot-averaging} is based on ideas similar to those of \citet{woodworth2020local}, whereby the algorithm may access the stochastic oracles via \rqga{} (\Cref{alg:quad-grad-access}). However, additional care must be taken to account for the fact that \Cref{alg:quad-grad-access} supplies stochastic gradient estimates of the \emph{quadratic subproblems} $Q_\lambda(u)$ as per the oracles models described in Assumptions \ref{assump:a1} and \ref{assump:a2}. Thus, the estimates---based in part on stochastic Hessian-vector products---have variance that scales with the norm of the respective iterates of \Cref{alg:one-shot-averaging} (see Assumption \ref{assump:a2}(b.ii)).

We also note two possible cases for the oracle access: Case 1 (\texttt{Different-Samples}) in \Cref{alg:quad-grad-access} requires both a call to a stochastic first-order oracle (which draws $z\sim \mathcal{D}$) and a call to a stochastic Hessian-vector product oracle (which draws a different $z'\sim\mathcal{D}$); while Case 2 (\texttt{Same-Sample}) allows both stochastic estimators to be observed for the same random sample $z\sim \mathcal{D}$. These cases differ by only a small constant factor in the final convergence rate, and we base our practical method on this single sample model. We refer the reader to \Cref{sec:comp_cost} for additional discussion of these alternative settings. Finally, having analyzed 
Algorithms \ref{alg:stochastic-newton}, \ref{alg:constrained-quadratic}, \ref{alg:one-shot-averaging}, and \ref{alg:quad-grad-access}, 
we may put them all together to provide our main theoretical result, whose proof can be found in \Cref{app:combined-alg}.
\begin{restatable}{theorem}{combinedalg}\label{thm:combined-alg}
Let $F$ satisfy Assumption~\ref{assump:a2}.
Then, for $K \geq 175$ and $R \geq \tilde{\Omega}(1)$, and for hyperparameters $T$, $\beta$, $\rr$, $\bar{\xi}$, $\lmin$, $N$, $C$ and parameter functions $\eta_k(\lambda)$, $w_k(\lambda)$ as in \Cref{tab:hyperparams,tab:paramfuncs}, the output of \fedsn{} (\Cref{alg:stochastic-newton}) with initial point $x_0 \in \R^d$, using Algorithms \ref{alg:constrained-quadratic}, \ref{alg:one-shot-averaging}, and \ref{alg:quad-grad-access} (for both Cases 1 and 2) satisfies
\[
\E [F(x_T)] - F^* \leq HB^2\prn*{\exp\prn*{-\frac{R}{\tilde{O}\prn*{\alpha B}}} + \exp\prn*{ -\frac{K}{O(1)}}} + \tilde{O}\prn*{\frac{\sigma B}{\sqrt{MK}} + \frac{HB^2}{KR} + \frac{\rho B^2}{\sqrt{K}R}}\ ,
\]
where $\tilde{\Omega}$, $\tilde{O}$ hide terms logarithmic in $R$, $K$, and $\alpha B$.
\end{restatable}

% \vspace{-0.5em}
\section{Comparison with related methods and lower bounds}\label{sec:compare}

In this section, we compare our algorithm's guarantees with those of other, previously proposed first-order methods. A difficulty in making this comparison is determining the ``typical'' relative scale of the parameters $H$, $\sigma$, $Q$, $\alpha$,  and $\rho$. We therefore draw inspiration from training generalized linear models in considering a natural scaling of the parameters that arises when the objective has the form $F(x) = \E_{z}\ell(\inner{x}{z})$, where $\abs{\ell'}$, $\abs{\ell''}$, and $\abs{\ell'''}$ are $O(1)$; this holds, e.g., for logistic regression problems (see \eqref{eq:logistic-regression}).
In this case, upper bounds on the derivatives of $F$ will generally scale with the norm of the data, i.e., $\nrm{z}$. So if we assume that $\nrm{z} \leq D$ for some $D$, then the derivatives of $F$ would generally scale as $\nrm{\nabla F(x)} \lesssim D$, $\nrm{\nabla^2 F(x)}_{\textrm{op}} \lesssim D^2$, and $\nrm{\nabla^3 F(x)}_{\textrm{op}} \lesssim D^3$, where $\nrm*{\cdot}_{\textrm{op}}$ denotes the operator norm. Thus, in the rest of this section, we will take $H=D^2$, $\sigma = D$, $U = D^3$, $\alpha = D$, and $\rho = D^2$. It is, of course, possible that these parameters could have different relationships, but we focus on this regime for simplicity.

In addition to working within this natural scaling, we consider the case where we have access to a sufficient number of machines (i.e., $M \gtrsim \frac{KR^3}{D^2B^2}$) and where $K$ is sufficiently large. We explore various regimes both w.r.t.\ the number of rounds of communication $R$ (though we restrict ourselves to $R \gtrsim DB\log(DBK)$ for technical reasons), as well as w.r.t.\ $DB$, which roughly captures the ``size" of the problem.
Thus, ignoring constants and terms logarithmic in $R$, $K$, and $\alpha B$ for the sake of clarity, our upper bound from \Cref{thm:combined-alg} reduces to
\begin{equation}
\E F(\hat{x}) - F^* \lesssim D^2B^2\exp\prn*{-\frac{R}{DB}} + D^2B^2\exp\prn*{-K} + \frac{D^2B^2}{KR^{3/2}} + \frac{D^2B^2}{KR} + \frac{D^2B^2}{\sqrt{K}R} \approx \frac{D^2B^2}{\sqrt{K}R}\ .
\end{equation}
% \vspace{-0.8em}
\subsubsection*{Comparison with Local SGD}
Recall that, under third-order smoothness assumptions \citep{yuan2020federated}, Local SGD converges at a rate of
\begin{equation*}
    \E F(\hat{x}) - F^* \leq \tilde{O}\prn*{\frac{HB^2}{KR} + \frac{\sigma B}{\sqrt{MKR}} + \frac{Q^{1/3}\sigma^{2/3}B^{5/3}}{K^{1/3}R^{2/3}}}\ .
\end{equation*}
For the setting as outlined above, this bound reduces to 
\begin{equation*}
    \E F(\hat{x}) - F^* \lesssim \frac{D^2B^2}{KR} + \frac{D^2B^2}{KR^2} +  \frac{D^{5/3}B^{5/3}}{K^{1/3}R^{2/3}} \approx \frac{D^2B^2}{KR} + \frac{D^{5/3}B^{5/3}}{K^{1/3}R^{2/3}}\ .
\end{equation*}
In the case where $DB$ is not too large ($DB \lesssim K^2R$), the dominant term for Local SGD is $\frac{D^{5/3}B^{5/3}}{K^{1/3}R^{2/3}}$, and so we see that our algorithm improves upon Local SGD as long as $R\gtrsim \frac{D^{1/3}B^{1/3}}{K^{1/6}}$.

% \vspace{-0.8em}
\subsubsection*{Comparison with \fedac{}}
The previous best known first-order distributed method under third-order smoothness assumptions is \fedac{} \citep{yuan2020federated}, an accelerated variant of the Local SGD method, which, as we recall, achieves a guarantee of
\begin{equation*}
    \E F(\hat{x}) - F^* \leq \tilde{O}\prn*{\frac{HB^2}{KR^2} + \frac{\sigma B}{\sqrt{MKR}} + \frac{H^{1/3}\sigma^{2/3}B^{4/3}}{M^{1/3}K^{1/3}R} + \frac{Q^{1/3}\sigma^{2/3}B^{5/3}}{K^{1/3}R^{4/3}}}\ .
\end{equation*}

For the setting as outlined above, this bound reduces to
\begin{equation*}
		\E F(\hat{x}) - F^* \lesssim \frac{D^2B^2}{KR^2} + \frac{D^{5/3}B^{5/3}}{K^{1/3}R^{4/3}}\ .
\end{equation*}
In the case where $DB$ is not too large ($DB \lesssim K^2R^2$), the dominant term for \fedac{} is $\frac{D^{5/3}B^{5/3}}{K^{1/3}R^{4/3}}$, and so we see that our algorithm improves upon \fedac{} as long as $R\lesssim \frac{\sqrt{K}}{DB}$, whereas for $R \gtrsim \frac{\sqrt{K}}{DB}$, \fedac{} provides better guarantees than \fedsn{}.

% \vspace{-0.8em}
\subsubsection*{Comparison with min-max method under Assumption \ref{assump:a1}}
\cite{woodworth2021min} identified the min-max optimal (up to logarithmic factors) stochastic first-order method in the distributed setting we consider, and under Assumption \ref{assump:a1}.  Namely, a min-max optimal method can be obtained by combining Minibatch-Accelerated-SGD \citep{cotter2011better}, which enjoys a guarantee of
\begin{equation*}
    \E F(\hat{x}) - F^* \leq O\prn*{\frac{HB^2}{R^2} + \frac{\sigma B}{\sqrt{MKR}}}\ ,
\end{equation*}
with Single-Machine Accelerated SGD, which runs $KR$ steps of an accelerated variant of SGD known as AC-SA \citep{lan2012optimal}, and enjoys a guarantee of
\begin{equation*}
    \E F(\hat{x}) - F^* \leq O\prn*{\frac{HB^2}{K^2R^2} + \frac{\sigma B}{\sqrt{KR}}}\ .
\end{equation*}

Therefore, an algorithm which returns the output of Minibatch Accelerated SGD when $K \leq \frac{\sigma^2 R^3}{H^2B^2}$, and otherwise returns the output of Single-Machine Accelerated SGD, achieves a guarantee of
\begin{equation*}
    \E F(\hat{x}) - F^* \leq O\prn*{\frac{H^2B^2}{K^2R^2} + \frac{\sigma B}{\sqrt{MKR}} + \min\crl*{\frac{HB^2}{R^2}, \frac{\sigma B}{\sqrt{KR}}}}\ .
\end{equation*}
As shown by \cite{woodworth2021min}, this matches the lower bound for stochastic distributed first-order optimization under Assumption \ref{assump:a1}, up to $O(\log^2 M)$ factors.

For the setting as outlined above, this bound reduces to
\begin{equation*}
		\E F(\hat{x}) - F^* \lesssim \frac{D^2B^2}{K^2R^2} + \frac{D^2B^2}{KR^2} + \min\crl*{\frac{D^2B^2}{R^2}, \frac{DB}{\sqrt{KR}}} \approx \frac{D^2B^2}{K^2R^2} +  \min\crl*{\frac{D^2B^2}{R^2}, \frac{DB}{\sqrt{KR}}}\ .
\end{equation*}

Thus, for $DB \lesssim \frac{R^{3/2}}{\sqrt{K}}$, the dominant term is $\frac{D^2B^2}{R^2}$, in which case \fedsn{} is better as long as $R \lesssim \sqrt{K}$. Furthermore, for $\frac{R^{3/2}}{\sqrt{K}} \lesssim DB \lesssim K^{3/2}R^{3/2}$, the dominant term is $\frac{DB}{\sqrt{KR}}$, in which case \fedsn{} is better as long as $R \gtrsim D^2B^2$. In these regimes, we see that Assumption \ref{assump:a2} allows \fedsn{} to improve over the best possible when relying only on Assumption \ref{assump:a1}.  We also note that when $DB \gtrsim K^{3/2}R^{3/2}$, the combined algorithm described above is better than the guarantee we prove for \fedsn{}---we do not know whether this is a weakness of our analysis or represents a true deficiency of \fedsn{} in this regime.

% \vspace{-0.8em}
\subsubsection*{Comparison with first-order lower bounds}
\cite{woodworth2021min} additionally provide lower bounds under other smoothness conditions, including quasi-self-concordance, which are relevant to the current work. Roughly speaking, they show that under Assumption \ref{assump:a2}(a), no first-order intermittent communication algorithm can guarantee suboptimality less than (ignoring constant and $\log M$ factors)
\begin{equation*}
   \E F(\hat{x}) - F^* \geq 
   \frac{HB^2}{K^2R^2} + \frac{\sigma B}{\sqrt{MKR}} + \min\crl*{\frac{HB^2}{R^2},\frac{\alpha\sigma B^2}{\sqrt{K}R^2},\frac{\sigma B}{\sqrt{KR}}}\ .
\end{equation*}
In the same parameter regime as above, the lower bound reduces to
\begin{equation*}
\E F(\hat{x}) - F^* \gtrsim \frac{D^2B^2}{K^2R^2} + \min\crl*{\frac{D^2B^2}{\sqrt{K}R^2},\frac{DB}{\sqrt{KR}}}\ .
\end{equation*}
Comparing this lower bound with our guarantee in \Cref{thm:combined-alg}, we see that, when $DB = O(1)$ and the number of rounds of communication is small (e.g., $R = O(\log K)$), our approximate Newton method can (ignoring $\log K$ factors) achieve an upper bound of $\E F(\hat{x}) - F^* \lesssim 1/\sqrt{K}$.  Therefore, in this important regime, \fedsn{} matches the lower bound under Assumption \ref{assump:a2}(a), albeit using a stronger oracle---the stochastic Hessian-vector product oracle ensured by Assumption \ref{assump:a2}(b), rather than only a stochastic first-order oracle.  No prior work has been able to match this lower bound, and so we do not know whether such an oracle (or perhaps allowing multiple accesses to each component, which is an even stronger oracle) is necessary in order to achieve it, or if perhaps the stronger oracle even allows for breaking it.  Either way, we improve over all prior stochastic methods we are aware of in this setting, and at the very least match the best possible using any stochastic first-order method.

\section{Experiments}\label{sec:experiments}

We begin by presenting a more practical variant of \fedsn{} called \fedsnlite{} in \Cref{alg:stochastic-newton-lite}. It differs from \fedsn{} in two major ways: first, it directly uses \ossgd{} (\Cref{alg:one-shot-averaging}) (with $\bar{\xi} = 1$) as an approximate quadratic solver without requiring a search over the regularization parameter as in \Cref{alg:constrained-quadratic}; and in addition, it scales the Newton update with a stochastic approximation of the Newton decrement (see \cite{nesterov1998introductory,boyd2004convex}) and a constant stepsize $\nu$ (we use $\nu=1.25$ in our experiments). Because of these changes, the Newton update is fairly robust to the choice of $\nu$, so we then only need to tune the learning rate of \ossgd{}, making it as usable as any other first-order method.  
\begin{algorithm}[H]
   \caption{\fedsnlite{}$(x_0)$}
   \label{alg:stochastic-newton-lite}
\begin{algorithmic}
\STATE \hspace{-1em}(Operating on objective $F(\cdot)$ with stochastic gradient $g(\cdot;\cdot)$ and Hessian-vector product $h(\cdot;\cdot,\cdot)$ oracles.)
    \REQUIRE $x_0 \in \R^d$.
    \renewcommand{\algorithmicrequire}{\textbf{Hyperparameters:}}
    \REQUIRE $T$: main iterations; $\nu$: Newton stepsize scale; and $\beta$: momentum.
   \FOR{$t = 0,1,\dots,T-1$}
\STATE \ossgd{}$(x_t,0)$
   \STATE $\nu_t := \nu\left(1+(\updatet_t^\top h(x_t, \updatet_t;z))^{1/2}\right)^{-1}$\ \hfill $\triangleright\ h(x_t, \updatet_t;z)\ \text{is s.t.}\ \mathbb{E}_z[h(x_t, \updatet_t;z)] = \nabla^2F(x_t)\updatet_t$ \footnotemark
   \STATE Update: $x_{t+1} = x_t + \nu_t\updatet_t$
   \ENDFOR
   \STATE \textbf{Return:} $x_T$
\end{algorithmic}
\end{algorithm}
\footnotetext{We can calculate an approximation to the Newton decrement at $\updatet_t$, i.e., $(\updatet_t^\top \nabla^2 F(x_t)\updatet_t)^{1/2}$,  using a call to the Hessian-vector product oracle (see Assumption \ref{assump:a2}(b)), along with an additional dot product, namely $\updatet_t^\top h(x_t, \updatet_t;z)$.}

Note that in each step of \fedsnlite{}, the subroutine \ossgd{} approximately solves a quadratic subproblem using a variant of one-shot averaging. Moreover, we implement \fedsnlite{} such that each call to the stochastic oracle from within \ossgd{} uses only a single sample (Case 2 in \Cref{alg:quad-grad-access}), and so it is asymptotically as expensive as the gradient oracle (i.e., $O(d)$, if $d$ is the dimension of the problem). For a discussion of the computational cost of using a Hessian-vector product oracle, see \Cref{sec:comp_cost}.

\textbf{Baselines.} We compare \fedsnlite{} against the two variants of \fedac{} (\Cref{alg:fedac}, \cite{yuan2020federated}), Minibatch SGD (\Cref{alg:mbsgd}, \cite{dekel2012optimal}), and Local SGD (\Cref{alg:lsgd}, \cite{zinkevich2010parallelized}). We also study the effect of adding Polyak's momentum, which we denote by $\beta$, to these algorithms (see \Cref{sec:baselines}). \fedac{} is mainly presented and analyzed for strongly convex functions by \cite{yuan2020federated}. In fact, they assume the knowledge of the strong convexity constant to tune \fedac{}, which is typically hard to know unless the function is explicitly regularized. To handle general convex functions \cite{yuan2020federated} build some \textit{internal regularization} into \fedac{} (see Appendix E.2 in their paper). However, their hyperparameter recommendations in this setting also depend on unknowns such as the smoothness of the function and the variance of the stochastic gradients. This poses a difficulty in comparing \fedac{} to the other algorithms, which do not require the knowledge of these unknowns. 

The main goal of our experiments is to be comprehensive and to fairly compare with the baselines, giving especially \fedac{} the benefit of the doubt in light of the above difficulty in comparing against it. With this in mind we conduct the two experiments on the  \href{https://www.csie.ntu.edu.tw/~cjlin/libsvmtools/datasets/binary.html\#a9a}{LIBSVM a9a} \citep{CC01a, Dua:2019} dataset for binary classification using the logistic regression model.

\subsubsection*{Experiment 1: Adding internal regularization to \fedac{}}\label{subsec:experiment1} 

\begin{figure}
     \centering
     \begin{subfigure}{\textwidth}
         \centering
        \includegraphics[width=0.9\textwidth]{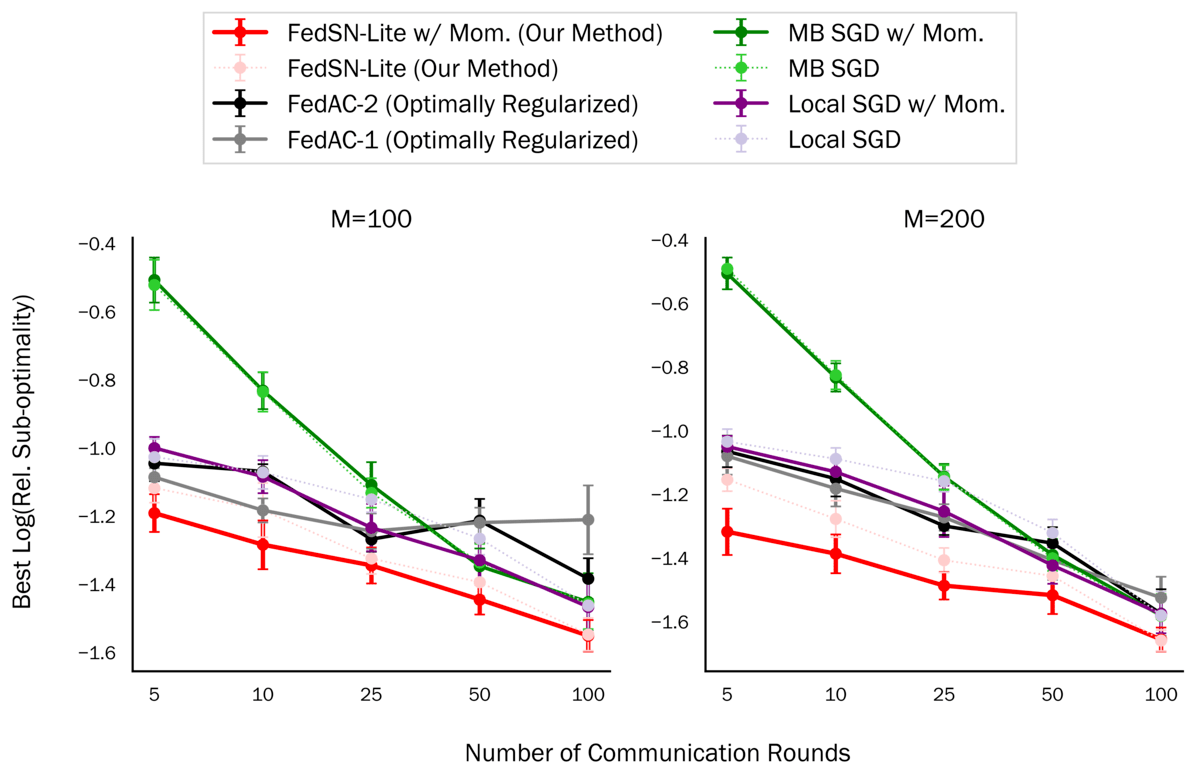}
        \caption{Each algorithm besides \fedac{} solves an unregularized ERM problem, and reports the best relative sub-optimality w.r.t.\ the optimal minimizer. All optimization runs were repeated $70$ times with the tuned hyperparameters.}
        \label{fig:a9a_grid_100_small}
     \end{subfigure}
     
     \begin{subfigure}{\textwidth}
         \centering
        \includegraphics[width=0.9\textwidth]{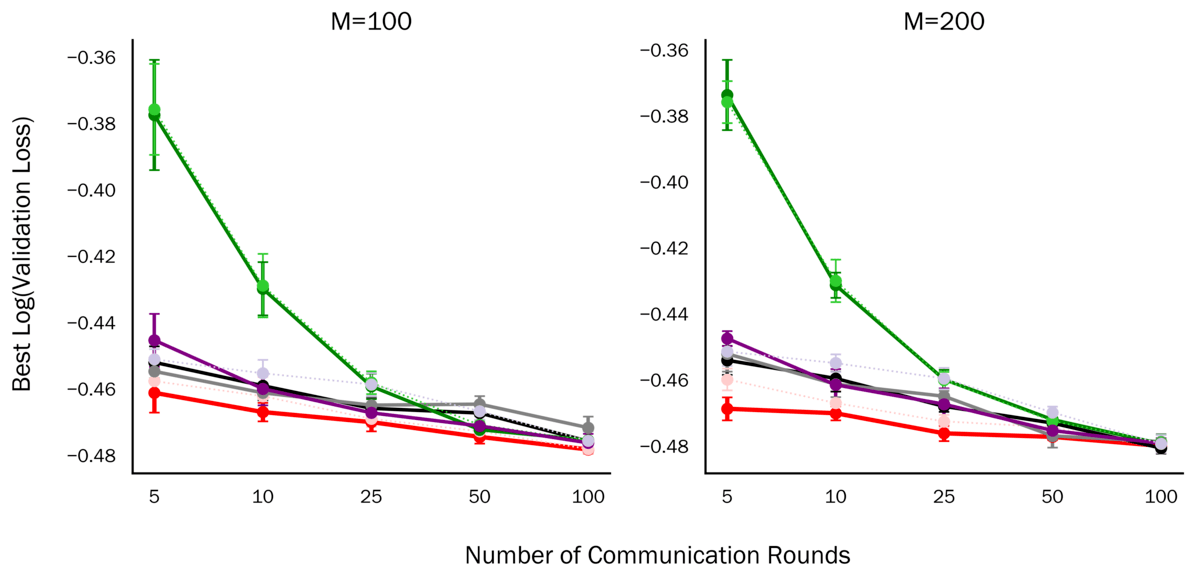}
        \caption{Each algorithm besides \fedac{} solves an unregularized ERM problem (with a single pass on the training data), and reports the best validation loss on a held-out dataset. We used $20{,}000$ samples for training, and validated on the remaining dataset. All optimization runs were repeated $20$ times with the tuned hyperparameters.}
        \label{fig:a9a_grid_100_test_small}
     \end{subfigure}
        
    \caption{\small Empirical comparison of \fedsnlite{} (\Cref{alg:stochastic-newton-lite}) to other methods (see \Cref{sec:baselines} for complete details) on the  \href{https://www.csie.ntu.edu.tw/~cjlin/libsvmtools/datasets/binary.html\#a9a}{LIBSVM a9a} \citep{CC01a, Dua:2019} dataset for minimizing:  (a) in-sample, and (b) out-of-sample unregularized logistic regression loss using $M \in \{100, 200\}$ machines. We vary the frequency of communication (horizontal axis of each plot), while keeping the total number of steps on each machine (theoretical parallel runtime) fixed at $KR=100$.  For \fedac{} we tune the learning rate for different levels of \textit{internal regularization} in~ $\{$1e-2, 1e-3, 1e-4, 1e-5, 1e-6$\}$. For the other algorithms we tune the learning rate for either $\beta=0$, i.e., without momentum, or for all $\beta\in \{0.1, 0.3, 0.5, 0.7, 0.9\}$, to choose the optimal level of momentum. We repeat this tuning procedure for each value of $M, R$, and thus each point in the plot represents an optimal configuration of that algorithm for that setting.}
\end{figure}

We take the more carefully optimized version of \fedac{} for strongly convex functions and tune its internal regularization and learning rate: (a) for minimizing the unregularized training loss (\Cref{fig:a9a_grid_100_small}), and (b) for minimizing the out-of-sample loss (\Cref{fig:a9a_grid_100_test_small}). This emulates the setting where the objective is assumed to be a general convex function, though \fedac{} sees a strongly convex function instead. Specifically, we are concerned with minimizing a convex function $F(x)$. In the first case (\Cref{fig:a9a_grid_100_small}) we are minimizing a finite sum, i.e., $F(x) = \sum_{i\in S} f_i(x)$ where $S$ is our training dataset. In the second case (\Cref{fig:a9a_grid_100_test_small}) $F(x)~= \mathbb{E}_{z\in \mathcal{D}}[f(x;z)]$, i.e., a stochastic optimization problem where we access the distribution $\mathcal{D}$ through our finite sample $S$. To estimate the true-error on $\mathcal{D}$, we split $S$ into $S_{train}$ and $S_{val}$, then sample without replacement from $S_{train}$ to train our models, and report the final performance on $S_{val}$.

\begin{figure}[H]
    \centering
    \hspace{-6em}
    \includegraphics[width=0.9\textwidth]{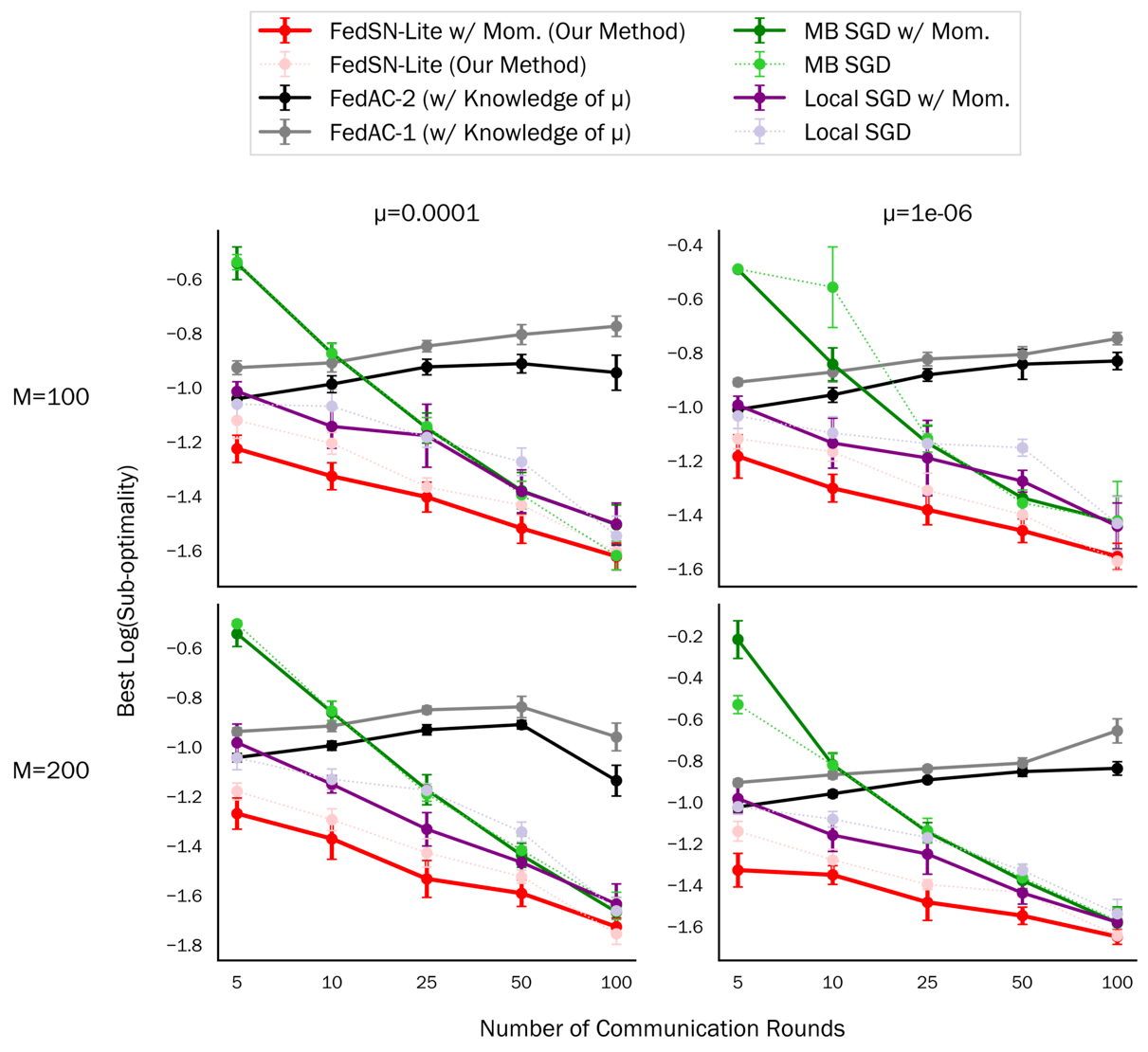}
    \caption{\small Empirical comparison of \fedsnlite{} (\Cref{alg:stochastic-newton-lite}) to other methods (see \Cref{sec:baselines} for complete details) on the  \href{https://www.csie.ntu.edu.tw/~cjlin/libsvmtools/datasets/binary.html\#a9a}{LIBSVM a9a} \citep{CC01a, Dua:2019} dataset for minimizing $\mu$-regularized logistic regression loss using $M \in \{100, 200\}$ machines and $\mu\in\{$1e-4, 1e-6$\}$ regularization strength. We vary the frequency of communication (horizontal axis of each plot), while keeping the total number of steps on each machine (theoretical parallel runtime) fixed at $KR=100$. Each algorithm besides solves a regularized ERM problem, and reports the best relative sub-optimality w.r.t.\ the optimal minimizer. For \fedac{} we use $\mu$ as the strong convexity constant to tune its hyperparameters. For the other algorithms we tune the learning rate for either $\beta=0$, i.e., without momentum, or for all $\beta\in \{0.1, 0.3, 0.5, 0.7, 0.9\}$, to choose the optimum level of momentum. We repeat this tuning procedure for each value of $\mu, M, R$, and thus each point in the plot represents an optimal configuration of that algorithm for that setting. All optimization runs were repeated $30$ times with the tuned hyperparameters.}
    \label{fig:a9a_grid_100_reg_small}
\end{figure}

\subsubsection*{Experiment 2: Optimizing regularized objectives}\label{subsec:experiment2}
We consider regularized empirical risk minimization (\Cref{fig:a9a_grid_100_reg_small}), where we provide \fedac{} the regularization strength $\mu$ which serves as its estimate of the strong convexity for the objective. Unlike Experiment 1, here we report regularized training loss, i.e., we train the models to optimize $F(x) + r(x) = \sum_{i\in S}f_i(x) + \frac{\mu}{2}\norm{x}_2^2$, for a finite dataset $S$. We also vary the regularization strength $\mu$, to understand the algorithms' dependence on the problem's conditioning. This was precisely the experiment conducted by \cite{yuan2020federated} (c.f., Figures 3 and 5 in their paper).

\textbf{Implementation details.} Both the experiments are performed over a wide range of communication rounds and machines, to mimic different systems trade-offs. We always search for the best learning rate $\eta \in \{0.0001,0.0002,$ $0.0005, 0.001, 0.002, 0.005, 0.01, 0.02, 0.05, 0.1, 0.2, 0.5, 1, 2, 5, 10, 20\}$ for every configuration of each algorithm (where we then set the parameter functions as $\eta(\lambda) = \eta$ and $w(\lambda) = \frac{1}{K}$). We verify that the optimal learning rates always lie in this range. More details can be found in \Cref{sec:implementation_details}.

\textbf{Observation.} In all our experiments we notice that \fedsnlite{} is either competitive with or outperforms the other baselines. This is especially true for the sparse communication settings, which are of most practical interest. A more comprehensive set of experiments can be found in \Cref{sec:add_exps}.

\section{Conclusion}
In this work, we have shown how to more efficiently optimize convex quasi-self-concordant objectives by leveraging parallel methods for quadratic problems. Our method can, in some parameter regimes, improve upon existing stochastic methods while maintaining a similar computational cost, and we have further seen how our method may provide empirical improvements in the low communication regime.  In order to do so we rely on stochastic Hessian-vector product access, instead of just stochastic gradient calculations.  In many situations stochastic Hessian-vector products can be calculated just as easily as stochastic gradients.  Furthermore, Hessian-vector products can be calculated to arbitrary accuracy using two stochastic gradient accesses {\em to the same component} (i.e., the same sample).  It remains open whether the same guarantees we achieve here can also be achieved using only independent stochastic gradients (a single stochastic gradient on each sample), or whether in the distributed stochastic setting access to Hessian-vector products is strictly more powerful than access to only independent stochastic gradients.

\paragraph{Acknowledgements.}
Research was partially supported by NSF-BSF award 1718970. BW is supported by a Google Research PhD Fellowship.

\bibliography{bibliography}
\clearpage 
%%%%%%%%%%%%%%%%%%%%%%%%%%%%%%%%%%%%%%%%%%%%%%%%%%%%%%%%%%%%

\appendix
\onecolumn

\section{Analysis of \Cref{alg:stochastic-newton}}\label{app:stochastic-newton}

We will use the following notation in our analysis:
\begin{equation}
Q^\sigma_t(\Delta x) = \frac{\sigma}{2} \Delta x^\top \nabla^2 F(x_t)\Delta x + \nabla F(x_t)^\top \Delta x.
\end{equation}

\begin{lemma}[Lemma 5 \citep{karimireddy2018global}]\label{lem:qsc-upper-lower-bound}
Let $F$ be $\delta(r)$-locally stable for a given $\delta : \R^+ \mapsto \R^+$, let $\nrm{x^*} \leq B$, let $r > 0$, let $\gamma = r / B$ and $x_\gamma^* = (1-\gamma)x_t + \gamma x^*$, and let $x_{t+1} = x_t + \updatet_t$ for $\nrm{\updatet_t}\leq r$. Then
\begin{align*}
F(x_{t+1}) - F(x_t) &\leq Q_t^{\delta(r)}(\updatet_t) \\
F(x_\gamma^*) - F(x_t) &\geq Q_t^{1/\delta(r)}(x_\gamma^* - x_t).
\end{align*}
\end{lemma}

\begin{lemma}[Lemma 6 \citep{karimireddy2018global}]\label{lem:quadratic-superscript-stuff}
For any convex domain $\mc{Q}$ and constants $a \cdot b \geq 1$,
\[
\min_{\Delta x \in \mc{Q}} Q^a(\Delta x) \leq \frac{1}{ab} \min_{\Delta x \in \mc{Q}} Q^{1/b}(\Delta x).
\]
\end{lemma}

We will now prove \Cref{lem:qsc-newton} from \Cref{sec:results}.

\newtonanalysis*
\begin{proof}
To begin, we are given by the assumptions of the theorem statement that \Cref{alg:stochastic-newton} chooses the update $\updatet_t$ such that $\nrm{\updatet_t} \leq 5\ccvar$ and
\begin{equation}
\E Q_t^{\delta(5\ccvar)}(\updatet_t) - \min_{\nrm{\Delta x} \leq \frac{1}{2}\ccvar} Q_t^{\delta(5\ccvar)}(\Delta x) \leq \epsilon + \theta\prn*{Q_t^{\delta(5\ccvar)}(0) - \min_{\nrm{\Delta x} \leq \frac{1}{2}\ccvar} Q_t^{\delta(5\ccvar)}(\Delta x)}.
\end{equation}
Therefore, 
\begin{align}
\E\brk*{F(x_{t+1}) - F(x_t)} 
&\leq \E Q_t^{\delta(5\ccvar)}(\updatet_t) \\
&\leq \epsilon + (1-\theta)\min_{\Delta x\,:\,\nrm{\Delta x} \leq \frac{1}{2}\ccvar} Q_t^{\delta(5\ccvar)}(\Delta x) \\
&\leq \epsilon + \frac{1-\theta}{\delta(5\ccvar) \delta(\frac{1}{2}\ccvar)} \min_{\Delta x\,:\,\nrm{\Delta x} \leq \frac{1}{2}\ccvar} Q_t^{1/\delta(\frac{1}{2}\ccvar)}(\Delta x) \\
&\leq \epsilon + \frac{1-\theta}{\delta(5\ccvar) \delta(\frac{1}{2}\ccvar)} Q_t^{1/\delta(\frac{1}{2}\ccvar)}\prn*{\prn*{1 - \frac{\ccvar}{2B}}x_t + \frac{\ccvar}{2B}x^* - x_t} \\
&\leq \epsilon + \frac{1-\theta}{\delta(5\ccvar) \delta(\frac{1}{2}\ccvar)}\brk*{F\prn*{\prn*{1 - \frac{\ccvar}{2B}}x_t + \frac{\ccvar}{2B}x^*} - F(x_t)} \\
&\leq \epsilon + \frac{1-\theta}{\delta(5\ccvar) \delta(\frac{1}{2}\ccvar)}\brk*{\prn*{1 - \frac{\ccvar}{2B}}F\prn*{x_t} + \frac{\ccvar}{2B}F^* - F(x_t)} \\
&= \epsilon - \frac{\ccvar(1-\theta)}{2B\delta(5\ccvar) \delta(\frac{1}{2}\ccvar)}\brk*{F(x_t) - F^*}.
\end{align}
Here we used \Cref{lem:qsc-upper-lower-bound} for the first inequality, \Cref{lem:quadratic-superscript-stuff} for the third, \Cref{lem:qsc-upper-lower-bound} again for the fifth, and the convexity of $F$ for the sixth.

Rearranging, and unravelling the recursion, we conclude that
\begin{align}
\E\brk*{F(x_T) - F^*} 
&\leq \prn*{1 - \frac{\ccvar(1-\theta)}{2B \delta(5\ccvar) \delta(\frac{1}{2}\ccvar)}}^T\E\brk*{F(x_0) - F^*} + \epsilon\sum_{t=0}^T \prn*{1 - \frac{\ccvar(1-\theta)}{2B \delta(5\ccvar) \delta(\frac{1}{2}\ccvar)}}^t \\
&\leq \E\brk*{F(x_0) - F^*}\exp\prn*{-\frac{T\ccvar(1-\theta)}{2B\delta(5\ccvar) \delta(\frac{1}{2}\ccvar)}} + \frac{2B\delta(5\ccvar) \delta(\frac{1}{2}\ccvar)\epsilon}{\ccvar(1-\theta)}\ .
\end{align}
This completes the proof.
\end{proof}

\section{Proof of \Cref{lem:constrained-quadratic-correctness}}\label{app:constrained-quadratic}
Before we analyze \Cref{alg:constrained-quadratic}, we recall some key definitions. For a given $x\in\R^d$, we let 
\begin{equation}
Q(u) = \frac{\bar{\xi}}{2}u^\top \nabla^2 F(x) u + \nabla F(x)^\top u,
\end{equation}
and for a regularization penalty $\lambda$, we also define
\begin{equation}
Q_\lambda(u) = \frac{1}{2}u^\top (\bar{\xi}\nabla^2 F(x) + \lambda \boldI) u + \nabla F(x)^\top u.
\end{equation}
We use $u_\lambda^*$ to denote the (unique) minimizer of $Q_\lambda(u)$, and we use $r^*(\lambda) = \nrm{u_\lambda^*}$ to denote the norm of the minimizer. For $0\leq r \leq r^*(0)$, we also use $\lambda_r$ to denote the value of $\lambda$ such that $r^*(\lambda_r) = r$ (\Cref{lem:norm-monotone} below shows that $\lambda_r$ is unique).

\begin{lemma}[Lemmas 35 and 36 \citep{carmon2020acceleration}]\label{lem:exists-perfect-regularizer}\label{lem:norm-monotone}
For any $r$, there exists a unique $\lambda_r \geq 0$ such that 
\[
u^*_r = \argmin_{u:\nrm{u}\leq r} Q(u) = \argmin_{u\in\R^d} Q_{\lambda_r}(u) = -(\nabla^2 F(x) + \lambda_r \boldI)^{-1} \nabla F(x).
\]
Also, $\lambda_r$ is decreasing in $r$, and if $\lambda_r \neq 0$ then $\nrm{u^*_r} = r$.
\end{lemma}

\begin{lemma}\label{lem:lambda-star-upper-bound}
For $r \leq r^*(0)$, $\lambda_r \in \brk*{0, \frac{\nrm{b}}{r}}$.
\end{lemma}
% \collapse{
\begin{proof}
By \Cref{lem:exists-perfect-regularizer}, $\lambda \geq 0$. If $\lambda > 0$, then since $\nabla^2 F(x) \succeq 0$ ( by convexity of $F$), we have that
\begin{equation}
r = \nrm{u^*_\lambda} = \nrm{(\nabla^2 F(x) + \lambda \boldI)^{-1} \nabla F(x)} \leq \frac{\nrm{\nabla F(x)}}{\lambda}.
\end{equation}
Rearranging completes the proof.
\end{proof}

\begin{lemma}\label{lem:bigger-r-okay}
For any $r\geq 0$ let $\lambda \geq \gamma \geq 0$ such that $r^*(\lambda) = r$ and $r^*(\gamma) = 2r$. Then
\[
\E\brk*{Q_\gamma(u) - Q_\gamma(u^*_\gamma)} \leq \epsilon
\implies \E\brk*{Q(u) - Q(u^*_\lambda)} \leq \epsilon .
\]
Furthermore, for any $y$ and $\gamma$,
\[
\E\brk*{Q_\gamma(u) - Q_\gamma(u^*_\gamma)} \leq \epsilon
\implies \E\brk*{Q(u) - Q(y)} \leq \epsilon + \frac{\gamma \nrm*{y}^2}{2}.
\]
\end{lemma}
% \collapse{
\begin{proof}
For any $u$,
\begin{align}
\E\brk*{Q_\gamma(u) - Q(u^*_\lambda)}
&= \E \brk*{Q_\gamma(u) - Q_\gamma(u^*_\gamma)} + Q_\gamma(u^*_\gamma) - Q_\gamma(u^*_\lambda) + \frac{\gamma}{2}\nrm*{u^*_\lambda}^2 \\
&\leq \epsilon + Q_\gamma(u^*_\gamma) - Q_\gamma(u^*_\lambda) + \frac{\gamma}{2}\nrm*{u^*_\lambda}^2. \label{eq:bigger-r-lemma-eq1}
\end{align}
By the $\gamma$-strong convexity of $Q_\gamma$, the fact that either $\gamma = 0$ or $\nrm{u^*_\lambda} = r$ and $\nrm{u^*_\gamma} = 2r$, and the reverse triangle inequality,
\begin{equation}
Q_\gamma(u^*_\gamma) - Q_\gamma(u^*_\lambda) + \frac{\gamma}{2}\nrm*{u^*_\lambda}^2
\leq -\frac{\gamma}{2}\nrm*{u^*_\gamma - u^*_\lambda}^2 + \frac{\gamma r^2}{2} 
\leq -\frac{\gamma (2r - r)^2}{2} + \frac{\gamma r^2}{2} 
= 0.\label{eq:bigger-r-lemma-eq2}
\end{equation}
Combining \eqref{eq:bigger-r-lemma-eq2} with \eqref{eq:bigger-r-lemma-eq1} completes the first part of the proof. For the second part, we observe that 
\begin{equation}
\E\brk*{Q(u) + \frac{\gamma}{2}\nrm{u}^2 - Q(y)}
= \E\brk*{Q_\gamma(u) - Q_\gamma(u^*_\gamma)} + Q_\gamma(u^*_\gamma) - Q_\gamma(y) + \frac{\gamma}{2}\nrm{y}^2 
\leq \epsilon + \frac{\gamma \nrm{y}^2}{2}. 
\end{equation}
Above, we used that $Q_\lambda(u^*_\lambda) = \min_u Q_\lambda(u) \leq Q_\lambda(y)$.
\end{proof}

\begin{lemma}\label{lem:norm-growth-with-lambda}
Let $\lambda \geq \gamma \geq 0$. Then ,
\[
\frac{\lambda}{\gamma} \geq \frac{\nrm{u^*_\gamma}}{\nrm{u^*_\lambda}}.
\]
\end{lemma}
% \collapse{
\begin{proof}
By the definition of $u^*_\gamma = \argmin_u Q_\gamma(u)$ and $u^*_\lambda = \argmin_u Q_\lambda(u)$ and the $\gamma$- and $\lambda$-strong convexity of $Q_\gamma$ and $Q_\lambda$, respectively,
\begin{equation}
\frac{\lambda - \gamma}{2}\prn*{\nrm*{u^*_{\gamma}}^2 - \nrm*{u^*_{\lambda}}^2} = Q_\gamma(u^*_\lambda) - Q_\gamma(u^*_\gamma) + Q_\lambda(u^*_\gamma) - Q_\lambda(u^*_\lambda) \geq \frac{\lambda + \gamma}{2}\nrm*{u^*_\lambda - u^*_\gamma}^2.
\end{equation}
Therefore, by rearranging and applying the reverse triangle inequality
\begin{equation}
\frac{\lambda - \gamma}{\lambda + \gamma} 
\geq \frac{\nrm*{u^*_\gamma - u^*_\lambda}^2}{\nrm*{u^*_{\gamma}}^2 - \nrm*{u^*_{\lambda}}^2} 
\geq \frac{\prn*{\nrm*{u^*_\gamma} - \nrm*{u^*_\lambda}}^2}{\nrm*{u^*_{\gamma}}^2 - \nrm*{u^*_{\lambda}}^2} \\
= \frac{\nrm*{u^*_\gamma} - \nrm*{u^*_\lambda}}{\nrm*{u^*_{\gamma}} + \nrm*{u^*_{\lambda}}}\ .
\end{equation}
By \Cref{lem:norm-monotone}, since $\lambda \geq \gamma$, $\nrm*{u^*_\gamma} \geq \nrm*{u^*_\lambda}$. Therefore, rearranging this inequality completes the proof.
\end{proof}

\begin{lemma}\label{lem:chernoff}
Let $X \sim \textrm{Binomial}(C,p)$ with $p \geq \frac{3}{4}$. Then,
\[
\P\prn*{X > \frac{C}{2}} \geq 1 - \exp\prn*{-\frac{C}{8}}\ .
\]
\end{lemma}
\begin{proof}
This is a simple application of Hoeffding's inequality:
\begin{equation}
\P\prn*{\frac{1}{C} X \leq \frac{1}{2} \leq p - \frac{1}{4}} \leq \exp\prn*{-\frac{C}{8}}\ .
\end{equation}
\end{proof}

\begin{lemma}\label{lem:subproblem-indicators-correct-whp}
Let $x \in \R^d$ be as input to \Cref{alg:constrained-quadratic}, let $Q_{\lambda^{(i)}}(u) = \frac{1}{2}u^\top \prn*{\delta(5\rr)\nabla^2 F(x) + \lambda^{(i)}}u + \nabla F(x)^\top u$, let $Q_{\lambda^{(i)}}^* \defeq \min_u Q_{\lambda^{(i)}}(u)$, and for each $i$, let $\tilde{u}^{(i,1)},\dots,\tilde{u}^{(i,C)}$ be computed as in \Cref{alg:constrained-quadratic} such that
\[
\E Q_{\lambda^{(i)}}(\tilde{u}^{(i,c)}) - Q_{\lambda^{(i)}}^* \leq \epsilon(\lambda^{(i)}) \leq \frac{\lambda^{(i)}\prn*{r^*(\lambda^{(i)})^2 + \rr^2}}{800}\ .
\]
Then,
\begin{align*}
r^*(\lambda^{(i)}) \in \brk*{2\rr,\,3\rr} &\implies \P\prn*{\sum_{c=1}^C \mathbbm{1}\crl*{\nrm*{\tilde{u}^{(i,c)}} \in \brk*{\frac{3}{2}\rr,\,\frac{7}{2}\rr}} > \frac{C}{2}} \geq 1 - \exp\prn*{-\frac{C}{8}} \\
r^*(\lambda^{(i)}) \not\in \brk*{\rr,\,4\rr} &\implies \P\prn*{\sum_{c=1}^C \mathbbm{1}\crl*{\nrm*{\tilde{u}^{(i,c)}} \in \brk*{\frac{3}{2}\rr,\,\frac{7}{2}\rr}} \leq \frac{C}{2}} \leq \exp\prn*{-\frac{C}{8}} \\
r^*(\lambda^{(i)}) < 2\rr &\implies \P\prn*{\sum_{c=1}^C \mathbbm{1}\crl*{\nrm*{\tilde{u}^{(i,c)}} \leq \frac{5}{2}\rr} > \frac{C}{2}} \geq 1 - \exp\prn*{-\frac{C}{8}} \\
r^*(\lambda^{(i)}) > 3\rr &\implies \P\prn*{\sum_{c=1}^C \mathbbm{1}\crl*{\nrm*{\tilde{u}^{(i,c)}} > \frac{5}{2}\rr} > \frac{C}{2}} \geq 1 - \exp\prn*{-\frac{C}{8}}\ .
\end{align*}
\end{lemma}
\begin{proof}
By the $\lambda^{(i)}$-strong convexity of $Q_{\lambda^{(i)}}$, for each $c$,
\begin{equation}
\E \brk*{\frac{\lambda^{(i)}}{2}\nrm{\tilde{u}^{(i,c)} - u^*_{\lambda^{(i)}}}^2} \leq \E Q_{\lambda^{(i)}}(\tilde{u}^{(i,c)}) - Q_{\lambda^{(i)}}^* \leq \epsilon(\lambda^{(i)}).
\end{equation}
Therefore, by Markov's inequality, for each $c$
\begin{equation}
\P\prn*{\nrm{\tilde{u}^{(i,c)} - u^*_{\lambda^{(i)}}}^2 \geq \frac{1}{100}r^*(\lambda^{(i)})^2 + \frac{1}{100}\rr^2} \leq \frac{200\epsilon(\lambda^{(i)})}{\lambda^{(i)}\prn*{r^*(\lambda^{(i)})^2 + \rr^2}} \leq \frac{1}{4}\ .
\end{equation}
Furthermore, by the reverse triangle inequality,
\begin{equation}
\abs*{\nrm{\tilde{u}^{(i,c)}} - r^*(\lambda^{(i)})} \leq \nrm{\tilde{u}^{(i,c)} - u^*_{\lambda^{(i)}}} < \sqrt{\frac{1}{100}r^*(\lambda^{(i)})^2 + \frac{1}{100}\rr^2} \leq \frac{r^*(\lambda^{(i)}) + \rr}{10}\ .
\end{equation}
Therefore, for each $c$
\begin{equation}
\P\prn*{\frac{9}{10}r^*(\lambda^{(i)}) - \frac{1}{10}\rr \leq \nrm{\tilde{u}^{(i,c)}} \leq \frac{11}{10}r^*(\lambda^{(i)}) + \frac{1}{10}\rr} \geq \frac{3}{4}\ .
\end{equation}
We now consider several cases:

If $r^*(\lambda^{(i)}) \in \brk*{2\rr,\,3\rr}$, then 
\begin{align}
\P\prn*{\nrm{\tilde{u}^{(i,c)}} \in \brk*{\frac{3}{2}\rr,\,\frac{7}{2}\rr}}
\geq \P\prn*{\frac{17}{10}\rr \leq \frac{9}{10}r^*(\lambda^{(i)}) - \frac{1}{10}\rr \leq \nrm{\tilde{u}^{(i,c)}} \leq \frac{11}{10}r^*(\lambda^{(i)}) + \frac{1}{10}\rr \leq \frac{34}{10}\rr} 
\geq \frac{3}{4}.
\end{align}
Therefore, by \Cref{lem:chernoff}
\begin{equation}
r^*(\lambda^{(i)}) \in \brk*{2\rr,\,3\rr} \implies \P\prn*{\sum_{c=1}^C \mathbbm{1}\crl*{\nrm*{\tilde{u}^{(i,c)}} \in \brk*{\frac{3}{2}\rr,\,\frac{7}{2}\rr}} > \frac{C}{2}} \geq 1 - \exp\prn*{-\frac{C}{8}}\ .
\end{equation}

If $r^*(\lambda^{(i)}) > 4\rr$, then
\begin{equation}
\P\prn*{\nrm{\tilde{u}^{(i,c)}} \not\in \brk*{\frac{3}{2}\rr,\,\frac{7}{2}\rr}}
\geq \P\prn*{\frac{7}{2}\rr \leq \frac{9}{10}r^*(\lambda^{(i)}) - \frac{1}{10}\rr \leq \nrm{\tilde{u}^{(i,c)}}} 
\geq \frac{3}{4}\ .
\end{equation}
Similarly, if $r^*(\lambda^{(i)}) < \rr$, then 
\begin{equation}
\P\prn*{\nrm{\tilde{u}^{(i,c)}} \not\in \brk*{\frac{3}{2}\rr,\,\frac{7}{2}\rr}}
\geq \P\prn*{\nrm{\tilde{u}^{(i,c)}} \leq \frac{11}{10}r^*(\lambda^{(i)}) + \frac{1}{10}\rr \leq \frac{12}{10}\rr} 
\geq \frac{3}{4}\ .
\end{equation}
Therefore, by \Cref{lem:chernoff},
\begin{equation}
r^*(\lambda^{(i)}) \not\in \brk*{\rr,\,4\rr} \implies \P\prn*{\sum_{c=1}^C \mathbbm{1}\crl*{\nrm*{\tilde{u}^{(i,c)}} \in \brk*{\frac{3}{2}\rr,\,\frac{7}{2}\rr}} \leq \frac{C}{2}} \leq \exp\prn*{-\frac{C}{8}}\ .
\end{equation}

If $r^*(\lambda^{(i)}) < 2\rr$, then
\begin{equation}
\P\prn*{\nrm{\tilde{u}^{(i,c)}} \leq \frac{5}{2}\rr}
\geq \P\prn*{\nrm{\tilde{u}^{(i,c)}} \leq \frac{11}{10}r^*(\lambda^{(i)}) + \frac{1}{10}\rr \leq \frac{23}{10}\rr} 
\geq \frac{3}{4}\ .
\end{equation}
Therefore, by \Cref{lem:chernoff}
\begin{equation}
r^*(\lambda^{(i)}) < 2\rr \implies \P\prn*{\sum_{c=1}^C \mathbbm{1}\crl*{\nrm*{\tilde{u}^{(i,c)}} \leq \frac{5}{2}\rr} > \frac{C}{2}} \geq 1 - \exp\prn*{-\frac{C}{8}}\ .
\end{equation}

Finally, if $r^*(\lambda^{(i)}) > 3\rr$, then
\begin{equation}
\P\prn*{\nrm{\tilde{u}^{(i,c)}} > \frac{5}{2}\rr}
\geq \P\prn*{\frac{26}{10}\rr \leq \frac{9}{10}r^*(\lambda^{(i)}) - \frac{1}{10}\rr \leq \nrm{\tilde{u}^{(i,c)}}} 
\geq \frac{3}{4}\ .
\end{equation}
Therefore, by \Cref{lem:chernoff}
\begin{equation}
r^*(\lambda^{(i)}) > 3\rr \implies \P\prn*{\sum_{c=1}^C \mathbbm{1}\crl*{\nrm*{\tilde{u}^{(i,c)}} > \frac{5}{2}\rr} > \frac{C}{2}} \geq 1 - \exp\prn*{-\frac{C}{8}}\ .
\end{equation}
This completes the proof.
\end{proof}

\begin{lemma}\label{lem:good-lambda-in-set}
Let $N \geq 1 + \frac{5}{2}\log\frac{\nrm{\nabla F(x)}}{3\rr\lmin}$. Then, either $\lmin \geq \lambda_{3\rr}$ or $\exists \lambda \in \Lambda_1$ with $r^*(\lambda) \in \brk*{2\rr,3\rr}$.
\end{lemma}
\begin{proof}
In the case that $\lmin \leq \lambda_{2\rr}$, by \Cref{lem:lambda-star-upper-bound}, with $N \geq 1 + \frac{5}{2}\log\frac{\nrm{\nabla F(x)}}{3\rr\lmin}$,
\begin{equation}
\lmin \leq \lambda_{3\rr} \leq \frac{\nrm{\nabla F(x)}}{3\rr} \leq \lmin c^{N-1}.
\end{equation}
So, $\lambda_{3\rr}$ is between the largest and smallest elements in $\Lambda_1$. It follows that for some $1 \leq n \leq N$,
\begin{equation}
\lmin c^{n-1} \leq \lambda_{3\rr} \leq \lmin c^n.
\end{equation} 
Therefore, by \Cref{lem:norm-growth-with-lambda},
\begin{equation}
\frac{3}{2} = \frac{\lmin \prn*{\frac{3}{2}}^n}{\lmin \prn*{\frac{3}{2}}^{n-1}} \geq \frac{\lmin \prn*{\frac{3}{2}}^n}{\lambda_{3\rr}} \geq \frac{3\rr}{r^*(\lmin \prn*{\frac{3}{2}}^n)} \implies r^*(\lmin \prn*{\frac{3}{2}}^n) \geq 2\rr
\end{equation}
Finally, by \Cref{lem:norm-monotone}, $r^*(\lmin \prn*{\frac{3}{2}}^n) \leq 3\rr$, so $r^*(\lmin \prn*{\frac{3}{2}}^n) \in \brk{2\rr,3\rr}$ as claimed.
\end{proof}

\begin{lemma}\label{lem:projection-okay}
Let $\lambda$ satisfy $r^*(\lambda) \leq 4\rr$ and let $\tilde{u}$ be chosen so that
\[
\E Q_\lambda(\tilde{u}) - Q_\lambda^* \leq \epsilon(\lambda) \defeq \frac{\lambda(r^*(\lambda)^2 + \rr^2)}{800}\ .
\]
Then if $r^*(\lambda) \geq \rr$,
\[
\E Q\prn*{\min\crl*{1,\ \frac{5\rr}{\nrm{\tilde{u}}}}\tilde{u}} - \min_{u:\nrm{u}\leq\frac{1}{2}\rr} \leq \frac{1}{2}\prn*{Q(0) - \min_{u:\nrm{u}\leq \frac{1}{2}\rr} Q(u)} + \epsilon(\lambda).
\]
Otherwise,
\[
\E Q\prn*{\min\crl*{1,\ \frac{5\rr}{\nrm{\tilde{u}}}}\tilde{u}} - \min_{u:\nrm{u}\leq\frac{1}{2}\rr} \leq \frac{1}{2}\prn*{Q(0) - \min_{u:\nrm{u}\leq \frac{1}{2}\rr} Q(u)} + \epsilon(\lambda) + \frac{\lambda \rr^2}{8}\ .
\]
\end{lemma}
\begin{proof}
First, 
\begin{align}
\E& Q\prn*{\min\crl*{1,\ \frac{5\rr}{\nrm{\tilde{u}}}}\tilde{u}} - \min_{u:\nrm{u}\leq \frac{1}{2}\rr} Q(u) \nonumber\\
&= \E Q\prn*{\prn*{1 - \min\crl*{1,\ \frac{5\rr}{\nrm{\tilde{u}}}}}0 + \min\crl*{1,\ \frac{5\rr}{\nrm{\tilde{u}}}}\tilde{u}} - \min_{u:\nrm{u}\leq \frac{1}{2}\rr} Q(u)  \\
&\leq \E\brk*{\prn*{1 - \min\crl*{1,\ \frac{5\rr}{\nrm{\tilde{u}}}}} \prn*{Q(0) - \min_{u:\nrm{u}\leq \frac{1}{2}\rr} Q(u)} + \min\crl*{1,\ \frac{5\rr}{\nrm{\tilde{u}}}}\prn*{Q(\tilde{u}) - \min_{u:\nrm{u}\leq \frac{1}{2}\rr} Q(u)}}  \\
&\leq \E\brk*{\frac{\nrm{\tilde{u}}}{\nrm{\tilde{u}} + 5\rr}} \prn*{Q(0) - \min_{u:\nrm{u}\leq \frac{1}{2}\rr} Q(u)} + \E Q(\tilde{u}) - \min_{u:\nrm{u}\leq \frac{1}{2}\rr} Q(u) \\
&\leq \frac{\E\nrm{\tilde{u}}}{\E\nrm{\tilde{u}} + 5\rr} \prn*{Q(0) - \min_{u:\nrm{u}\leq \frac{1}{2}\rr} Q(u)} + \E Q(\tilde{u}) - \min_{u:\nrm{u}\leq \frac{1}{2}\rr} Q(u).
\end{align}
For the first inequality we used the convexity of $Q$, and for the final inequality we used Jensen's inequality on the concave function $x \mapsto \frac{x}{x+5\rr}$.
Next, we bound
\begin{align}
\E \nrm{\tilde{u}} 
&\leq r^*(\lambda) + \E\nrm{\tilde{u} - u^*_\lambda} \\
&\leq 4\rr + \sqrt{\E\nrm{\tilde{u} - u^*_\lambda}^2} \\
&\leq 4\rr + \sqrt{\frac{2}{\lambda}\E Q_\lambda(\tilde{u}) - Q_\lambda^*} \\
&\leq 4\rr + \sqrt{\frac{2\epsilon(\lambda)}{\lambda}} \\
&\leq 4\rr + \sqrt{\frac{r^*(\lambda)^2 + \rr^2}{400}} \\
&\leq 4\rr + \sqrt{\frac{17\rr^2}{400}} \\
&< 5\rr.
\end{align}
Therefore, 
\begin{equation}
\E Q\prn*{\min\crl*{1,\ \frac{5\rr}{\nrm{\tilde{u}}}}\tilde{u}} - \min_{u:\nrm{u}\leq \frac{1}{2}\rr} Q(u) 
\leq \frac{1}{2}\prn*{Q(0) - \min_{u:\nrm{u}\leq \frac{1}{2}\rr} Q(u)} + \E Q(\tilde{u}) - \min_{u:\nrm{u}\leq \frac{1}{2}\rr} Q(u).
\end{equation}
If $r^*(\lambda) \geq \rr$, then by the first part of \Cref{lem:bigger-r-okay}
\begin{equation}
\E Q(\tilde{u}) - \min_{u:\nrm{u}\leq \frac{1}{2}\rr} Q(u) \leq \E Q(\tilde{u}) - \min_{u:\nrm{u}\leq \frac{1}{2}r^*(\lambda)} Q(u) \leq \epsilon(\lambda).
\end{equation}
Otherwise, by the second part of \Cref{lem:bigger-r-okay}
\begin{equation}
\E Q(\tilde{u}) - \min_{u:\nrm{u}\leq \frac{1}{2}\rr} Q(u) = \E Q(\tilde{u}) - Q\prn{u^*_{\lambda_{\frac{1}{2}\rr}}} \leq \epsilon(\lambda) + \frac{\lambda\nrm{u^*_{\lambda_{\frac{1}{2}\rr}}}^2}{2} = \epsilon(\lambda) + \frac{\lambda \rr^2}{8}\ .
\end{equation}
\end{proof}

We are now ready to prove \Cref{lem:constrained-quadratic-correctness}.
\cqscorrectness*

\begin{proof}
First, we note that $\abs*{\Lambda_1} = N$, and in each iteration either the algorithm terminates or $\Lambda_{i+1}$ is chosen such that $\abs*{\Lambda_{i+1}} \leq \frac{1}{2}\abs*{\Lambda_i}$. Therefore, the algorithm terminates after at most $\ceil{\log_2 N}$ iterations.

By \Cref{lem:good-lambda-in-set}, either $\lmin \geq \lambda_{3\rr}$ or there exists $\lambda \in \Lambda_1$ for some $\lambda$ such that $r^*(\lambda) \in \brk*{2\rr,3\rr}$. If such a $\lambda$ exists, we denote this (not necessarily unique) value $\lambda^*$.

By \Cref{lem:subproblem-indicators-correct-whp} and the union bound, the following holds for all iteration $i = 1,2,\dots$ with probability at least $1 - 2\ceil{\log_2 N}\exp\prn*{-\frac{C}{8}}$:
\begin{align}
r^*(\lambda^{(i)}) \in \brk*{2\rr,\,3\rr} &\implies \sum_{c=1}^C \mathbbm{1}\crl*{\nrm*{\tilde{u}^{(i,c)}} \in \brk*{\frac{3}{2}\rr,\,\frac{7}{2}\rr}} > \frac{C}{2} \label{eq:constrained-quadratic-case1}\\
r^*(\lambda^{(i)}) \not\in \brk*{\rr,\,4\rr} &\implies \sum_{c=1}^C \mathbbm{1}\crl*{\nrm*{\tilde{u}^{(i,c)}} \in \brk*{\frac{3}{2}\rr,\,\frac{7}{2}\rr}} \leq \frac{C}{2} \label{eq:constrained-quadratic-case2}\\
r^*(\lambda^{(i)}) < 2\rr &\implies \sum_{c=1}^C \mathbbm{1}\crl*{\nrm*{\tilde{u}^{(i,c)}} \leq \frac{5}{2}\rr} > \frac{C}{2} \label{eq:constrained-quadratic-case3}\\
r^*(\lambda^{(i)}) > 3\rr &\implies \sum_{c=1}^C \mathbbm{1}\crl*{\nrm*{\tilde{u}^{(i,c)}} > \frac{5}{2}\rr} > \frac{C}{2}\ .\label{eq:constrained-quadratic-case4}
\end{align}
For most of the rest of the proof, we condition on this event, which we denote $E$.

Under $E$, if $\lambda^{(i)} = \lambda^*$, then the algorithm will terminate on Line \ref{algline:lambda-accepted}, and even if $\lambda^{(i)} \neq \lambda^*$, if the algorithm terminates on Line \ref{algline:lambda-accepted}, then $\lambda^{(i)} \in \brk*{\rr,4\rr}$. In either case, by the first part of \Cref{lem:projection-okay}
\begin{equation}
\E Q(\hat{u}) - \min_{u:\nrm{u}\leq\frac{1}{2}\rr} Q(u) \leq \frac{1}{2}\prn*{Q(0) - \min_{u:\nrm{u}\leq\frac{1}{2}\rr} Q(u)} + \epsilon(\lambda^{(i)}).
\end{equation}
Finally, since $r^*(\lambda^{(i)}) \leq 4\rr$, $\lambda^{(i)} \geq \lambda_{4\rr}$.

If the algorithm instead updates 
\begin{equation}
\Lambda_{i+1} = \Lambda_i \setminus \crl*{\lambda \in \Lambda_i\,:\, \lambda \geq \lambda^{(i)}}
\end{equation}
as on Line \ref{algline:lower}, then conditioned on $E$,
\begin{equation}
\sum_{c=1}^C \mathbbm{1}\crl*{\nrm*{\tilde{u}^{(i,c)}} \in \brk*{\frac{3}{2}\rr,\,\frac{7}{2}\rr}} \leq \frac{C}{2} < \sum_{c=1}^C \mathbbm{1}\crl*{\nrm*{\tilde{u}^{(i,c)}} \leq \frac{5}{2}\rr}
\end{equation}
implies $\lambda^{(i)} < 2\rr$. By \Cref{lem:norm-monotone}, since $r^*(\lambda^*) \geq 2\rr > r^*(\lambda^{(i)})$, $\lambda^* < \lambda^{(i)}$ and therefore $\lambda^* \in \Lambda_{i+1}$.

If the algorithm instead updates 
\begin{equation}
\Lambda_{i+1} = \Lambda_i \setminus \crl*{\lambda \in \Lambda_i\,:\, \lambda \leq \lambda^{(i)}}
\end{equation}
as on Line \ref{algline:higher}, then conditioned on $E$
\begin{equation}
\sum_{c=1}^C \mathbbm{1}\crl*{\nrm*{\tilde{u}^{(i,c)}} \in \brk*{\frac{3}{2}\rr,\,\frac{7}{2}\rr}} \leq \frac{C}{2} < \sum_{c=1}^C \mathbbm{1}\crl*{\nrm*{\tilde{u}^{(i,c)}} > \frac{5}{2}\rr}
\end{equation}
implies $\lambda^{(i)} > 3\rr$. By \Cref{lem:norm-monotone}, since $r^*(\lambda^*) \leq 3\rr < r^*(\lambda^{(i)})$, we have $\lambda^* > \lambda^{(i)}$, and therefore $\lambda^* \in \Lambda_{i+1}$.

Finally, $E$ implies that the algorithm will never reach Line \ref{algline:return-0}.

Therefore, conditioned on $E$, if $\lambda^* \in \Lambda_1$, then the algorithm will never remove $\lambda^*$ from the set of $\lambda$'s under consideration, and it will eventually terminate on Line \ref{algline:lambda-accepted} by returning a point such that 
\begin{equation}
\E Q(\hat{u}) - \min_{u:\nrm{u}\leq\frac{1}{2}\rr} Q(u) \leq \frac{1}{2}\prn*{Q(0) - \min_{u:\nrm{u}\leq\frac{1}{2}\rr} Q(u)} + \epsilon(\lambda^{(i)}).
\end{equation}
Otherwise, if such a $\lambda^*$ does not exist and the algorithm does not terminate on Line \ref{algline:lambda-accepted}, then it terminates on Line \ref{algline:use-lmin} using $\lmin \geq \lambda_{3\rr} \geq \lambda_{4\rr}$, which implies $r^*(\lmin) \leq 4\rr$. Therefore, by the second part of \Cref{lem:projection-okay}, 
\begin{equation}
\E Q(\tilde{u}) - \min_{u:\nrm{u}\leq \frac{1}{2}\rr}Q(u) \leq \frac{1}{2}\prn*{Q(0) - \min_{u:\nrm{u}\leq\frac{1}{2}\rr} Q(u)} + \epsilon(\lmin) + \frac{\lmin \rr^2}{8}\ .
\end{equation}
Therefore since $\epsilon(\lambda)$ is decreasing in $\lambda$, conditioned on $E$ the algorithm's output satisfies
\begin{equation}
\E Q(\tilde{u}) - \min_{u:\nrm{u}\leq \frac{1}{2}\rr}Q(u) \leq \frac{1}{2}\prn*{Q(0) - \min_{u:\nrm{u}\leq\frac{1}{2}\rr} Q(u)} + \epsilon(\lambda_{4\rr}) + \frac{\lmin \rr^2}{8}\ .
\end{equation}

We now consider the case that $E$ does not hold. In this case, the algorithm's output is guaranteed to have norm at most $5\rr$. Therefore, 
\begin{align}
Q(\hat{u}) = \frac{\delta(5\rr)}{2}\hat{u}^\top \nabla^2 F(x) \hat{u} + \nabla F(x)^\top \hat{u} &\leq \frac{25\delta(5\rr)\nrm{\nabla^2 F(x)}_2\rr^2}{2} + 5\rr\nrm{\nabla F(x)} \\
&= Q(0) + \frac{25\delta(5\rr)\nrm{\nabla^2 F(x)}_2\rr^2}{2} + 5\rr\nrm{\nabla F(x)}.
\end{align}
Therefore, conditioned on $\lnot E$,
\begin{equation}
Q(\hat{u}) - \min_{u:\nrm{u}\leq\frac{1}{2}\rr} Q(u) \leq Q(0) - \min_{u:\nrm{u}\leq\frac{1}{2}\rr} Q(u) + \frac{25\delta(5\rr)\nrm{\nabla^2 F(x)}_2\rr^2}{2} + 5\rr\nrm{\nabla F(x)}.
\end{equation}

We conclude by noting that
\begin{align}
\E Q(\hat{u})& - \min_{u:\nrm{u}\leq\frac{1}{2}\rr} Q(u) \nonumber\\
&= \E\brk*{Q(\hat{u})\,|\, E}\P(E) + \E\brk*{Q(\hat{u})\,|\, \lnot E}\P(\lnot E) - \min_{u:\nrm{u}\leq\frac{1}{2}\rr} Q(u) \\
&\leq \frac{1 + \P(\lnot E)}{2}\prn*{Q(0) - \min_{u:\nrm{u}\leq\frac{1}{2}\rr} Q(u)} + \epsilon(\lambda_{4\rr}) + \frac{\lmin \rr^2}{8} \\
&\qquad+ \prn*{\frac{25\delta(5\rr)\nrm{\nabla^2 F(x)}_2\rr^2}{2} + 5\rr\nrm{\nabla F(x)}}\P(\lnot E) \\
&\leq \frac{1 + 2\ceil{\log_2 N}\exp\prn*{-\frac{C}{8}}}{2}\prn*{Q(0) - \min_{u:\nrm{u}\leq\frac{1}{2}\rr} Q(u)} + \epsilon(\lambda_{4\rr}) + \frac{\lmin \rr^2}{8} \nonumber\\
&\qquad+ \prn*{\frac{25\delta(5\rr)\nrm{\nabla^2 F(x)}_2\rr^2}{2} + 5\rr\nrm{\nabla F(x)}}\cdot 2\ceil{\log_2 N}\exp\prn*{-\frac{C}{8}}
\end{align}
Using the fact that 
\begin{equation}
C \geq 8\log\prn*{\ceil{\log_2 N}\prn*{4 + \frac{8\delta(5\rr)\nrm{\nabla^2 F(x)}_2}{\lmin} + \frac{80\nrm{\nabla F(x)}}{\lmin \rr}}}
\end{equation}
we conclude
\begin{equation}
\E Q(\hat{u}) - \min_{u:\nrm{u}\leq\frac{1}{2}\rr} Q(u) \leq \frac{3}{4}\prn*{Q(0) - \min_{u:\nrm{u}\leq\frac{1}{2}\rr} Q(u)} + \epsilon(\lambda_{4\rr}) + \frac{\lmin \rr^2}{4}\ .
\end{equation}
This completes the proof.
\end{proof}

\section{Proof of \Cref{lem:sgd-with-multiplicative-noise}}\label{app:proof-sgd-with-multiplicative-noise}
\sgdwithmultnoise*
\begin{proof}
We will use $\bu_k = \frac{1}{M}\sum_{m=1}^M u_k^m$ to denote the average of each independent run of SGD's $k^{\textrm{th}}$ iterate. This quantity is never explicitly computed until the end, but we can nevertheless use it for our analysis. Likewise, we will use $\bg_k = \frac{1}{M}\sum_{m=1}^M \gamma(u_k^m;z_k^m,z_k^{'m})$ to denote the average of the stochastic gradients of $Q_\lambda(u)$ computed at time $k$, whereby we recall that $\gamma(u_k^m;z_k^m,z_k^{'m}) \defeq \bar{\xi}h(x,u_k^m;z_k^{'m})+\lambda u_k^m+g(x;z_k^m)$ as defined in \Cref{alg:one-shot-averaging}, along with the requisite oracle access as described in \Cref{alg:quad-grad-access}. 

We also have that, by Assumptions \ref{assump:a1}(c) and \ref{assump:a2}(b), for Case 1:
\begin{equation*}\E_{z_k^m,z_k^{'m}} [\gamma(u_k^m;z_k^m,z_k^{'m})] = \nabla Q_\lambda(u)\ \textrm{and}\ \E_{z_k^m,z_k^{'m}} \nrm*{\gamma(u_k^m;z_k^{m},z_k^{'m}) - \nabla Q_\lambda(u_k^m)}^2 \leq \sigma^2 + \rho^2\nrm*{u_k^m}^2,
\end{equation*}
while for Case 2 we have:
\begin{equation*}\E_{z_k^m} [\gamma(u_k^m;z_k^m,z_k^{m})] = \nabla Q_\lambda(u)\ \textrm{and}\ \E_{z_k^m} \nrm*{\gamma(u_k^m;z_k^{m},z_k^{m}) - \nabla Q_\lambda(u_k^m)}^2 \leq 2\sigma^2 + 2\rho^2\nrm*{u_k^m}^2,
\end{equation*}
where we used the fact that $\nrm*{a+b}^2 \leq 2\nrm*{a}^2+2\nrm*{b}^2$ for any $a,b \in \R^d$, and so in either case the variance term is bounded by $2\sigma^2 + 2\rho^2\nrm*{u_k^m}^2$.

A key feature of these stochastic gradients of $Q_\lambda (u)$, which we will use frequently, is that by the linearity of $\nabla Q_\lambda(u) = (\bar{\xi}\nabla^2F(x)+\lambda \boldI)u + \nabla F(x)$, 
\begin{equation}
\E \bg_k = \frac{1}{M}\sum_{m=1}^M \E \gamma(u_k^m;z_k^m,z_k^{'m}) = \frac{1}{M}\sum_{m=1}^M \E \nabla Q_\lambda(u_k^m) = \nabla Q_\lambda(\E \bu_k).
\end{equation}
We begin by expanding,
\begin{align}
\E&\nrm*{\bu_{k+1} - u^*}^2 \nonumber\\
&= \E\nrm*{\bu_k - u^*}^2 + \eta_k^2\E\nrm*{\bg_k}^2 - 2\eta_k\E\inner{\nabla Q_\lambda(\bu_k)}{\bu_k - u^*} \\
&= \E\nrm*{\bu_k - u^*}^2 + \eta_k^2\E\nrm*{\nabla Q_\lambda(\bu_k)}^2 - 2\eta_k\brk*{Q_\lambda(\bu_k) - Q_\lambda^* + \frac{\lambda}{2}\nrm*{\bu_k - u^*}^2} + \eta_k^2\E\nrm*{\bg_k - \nabla Q_\lambda(\bu_k)}^2 \\
&\leq (1 - \eta_k\lambda)\E\nrm*{\bu_k - u^*}^2 - 2\eta_k(1 - (\bar{\xi}H+\lambda)\eta_k)\E\brk*{Q_\lambda(\bu_k) - Q_\lambda^*} + \eta_k^2\prn*{\frac{2\sigma^2}{M} + \frac{2\rho^2}{M^2}\sum_{m=1}^M\E\nrm*{u_k^m}^2}.
\end{align}
Since $\eta_k \leq \frac{1}{2(\bar{\xi}H+\lambda)}$ for all $k$, $(1 - (\bar{\xi}H+\lambda)\eta_k) \geq \frac{1}{2}$ so we can rearrange
\begin{align}
\E\brk*{Q_\lambda(\bu_k) - Q_\lambda^*} \leq \prn*{\frac{1}{\eta_k} - \lambda}\E\nrm*{\bu_k - u^*}^2 - \frac{1}{\eta_k}\E\nrm*{\bu_{k+1} - u^*}^2 + \eta_k\prn*{\frac{2\sigma^2}{M} + \frac{2\rho^2}{M^2}\sum_{m=1}^M\E\nrm*{u_k^m}^2}. \label{eq:sgd-sc-quadratics-multiplicative-noise}
\end{align}
From here, we note that since $u_k^1,\dots,u_k^M$ are i.i.d.,
\begin{align}
\frac{1}{M^2}\sum_{m=1}^M\E\nrm*{u_k^m}^2 
&= \frac{1}{M^2}\sum_{m=1}^M\brk*{\E\nrm*{u_k^m - \E u_k^m}^2 + \nrm*{\E u_k^m}^2} \\
&= \E\nrm*{\bu_k^m - \E \bu_k^m}^2 + \frac{1}{M}\nrm*{\E \bu_k}^2 \\
&\leq \E\nrm*{\bu_k - u^*}^2 + \frac{1}{M}\nrm*{\E \bu_k}^2 \\
&\leq \E\nrm*{\bu_k - u^*}^2 + \frac{2}{M}\nrm*{\E \bu_k - u^*}^2 + \frac{2}{M}\nrm*{u^*}^2.
\end{align}
Furthermore, for each $k$, since $\eta_{k-1} \leq \frac{1}{\bar{\xi}H+\lambda}$
\begin{align}
\nrm*{\E \bu_k - u^*}^2
&\leq \nrm*{\E \bu_{k-1} - u^*}^2 + \eta_{k-1}^2\nrm*{\nabla Q_\lambda(\E\bu_{k-1})}^2 - 2\eta_{k-1}\inner{\nabla Q_\lambda(\E \bu_{k-1})}{\bu_{k-1} - u^*} \\
&\leq \nrm*{\E \bu_{k-1} - u^*}^2 + 2(\bar{\xi}H+\lambda)\eta_{k-1}^2\brk*{Q_\lambda(\E\bu_{k-1}) - Q_\lambda^*} - 2\eta_{k-1}\brk*{Q_\lambda(\E\bu_{k-1}) - Q_\lambda^*} \\
&\leq \nrm*{\E \bu_{k-1} - u^*}^2 \leq \nrm*{\E \bu_0 - u^*}^2 = \nrm{u^*}^2.
\end{align}
Therefore, returning to \eqref{eq:sgd-sc-quadratics-multiplicative-noise},
\begin{align}
\E\brk*{Q_\lambda(\bu_k) - Q_\lambda^*} 
&\leq \prn*{\frac{1}{\eta_k} - \lambda}\E\nrm*{\bu_k - u^*}^2 - \frac{1}{\eta_k}\E\nrm*{\bu_{k+1} - u^*}^2 + \frac{2\eta_k\sigma^2}{M} + 2\eta_k\rho^2\prn*{\E\nrm*{\bu_k - u^*}^2 + \frac{4}{M}\nrm*{u^*}^2} \\
&= \prn*{\frac{1}{\eta_k} - \lambda + \eta_k\rho^2}\E\nrm*{\bu_k - u^*}^2 - \frac{1}{\eta_k}\E\nrm*{\bu_{k+1} - u^*}^2 + \frac{2\eta_k\prn*{\sigma^2 + \rho^2\nrm*{u^*}^2}}{M}\ .
\end{align}

We first consider the case $K \leq \frac{2\max\crl*{\bar{\xi}H+\lambda,\frac{\rho^2}{\lambda}}}{\lambda}$, so that
\begin{equation*}\eta_k = \eta = \min\crl*{\frac{1}{2(\bar{\xi}H+\lambda)},\ \frac{\lambda}{2\rho^2}}\quad \textrm{ and }\quad w_k = (1-\lambda\eta + \eta^2\rho^2)^{-k-1}.
\end{equation*}
Then,
\begin{align}
\E &Q_\lambda\prn*{\frac{1}{\sum_{k=0}^K w_k}\sum_{k=0}^K w_k \bu_k} - Q_\lambda^* \nonumber\\
&\leq \frac{1}{\eta\sum_{k=0}^K w_k}\sum_{k=0}^K\brk*{ w_k\prn*{1 - \eta\lambda + \eta^2\rho^2}\E\nrm*{\bu_k - u^*}^2 - w_k\E\nrm*{\bu_{k+1} - u^*}^2} + \frac{2\eta\prn*{\sigma^2 + \rho^2\nrm*{u^*}^2}}{M} \\
&= \frac{1}{\eta\sum_{k=0}^K w_k}\sum_{k=0}^K\brk*{ \prn*{1 - \eta\lambda + \eta^2\rho^2}^{-k}\E\nrm*{\bu_k - u^*}^2 - \prn*{1 - \eta\lambda + \eta^2\rho^2}^{-(k+1)}\E\nrm*{\bu_{k+1} - u^*}^2}\\
&\qquad + \frac{2\eta\prn*{\sigma^2 + \rho^2\nrm*{u^*}^2}}{M} \\
&\leq \frac{\E\nrm*{\bu_0 - u^*}^2}{\eta\sum_{k=0}^K \prn*{1 - \eta\lambda + \eta^2\rho^2}^{-(k+1)}} + \frac{2\eta\prn*{\sigma^2 + \rho^2\nrm*{u^*}^2}}{M} \\
&\leq \frac{\nrm*{u^*}^2}{\eta\max\crl*{K,\,\prn*{1 - \eta\lambda + \eta^2\rho^2}^{-(K+1)}}} + \frac{2\eta\prn*{\sigma^2 + \rho^2\nrm*{u^*}^2}}{M} \\
&\leq 2\max\crl*{\bar{\xi}H+\lambda, \frac{\rho^2}{\lambda}}\nrm*{u^*}^2\min\crl*{\frac{1}{K},\,\exp\prn*{-\frac{K+1}{4}\min\crl*{\frac{\lambda}{\bar{\xi}H+\lambda},\ \frac{\lambda^2}{\rho^2}}}^{K+1}}\nonumber \\
&\qquad\qquad + \frac{2(\sigma^2 + \rho^2\nrm*{u^*}^2)}{\lambda MK}\ .\label{eq:sgd-with-multiplicative-noise-case1}
\end{align}
For the last line, we used that $K \leq \frac{2\max\crl*{\bar{\xi}H+\lambda,\frac{\rho^2}{\lambda}}}{\lambda} = \frac{1}{\eta\lambda}$.

In the second case $\prn*{K > \frac{2\max\crl*{\bar{\xi}H+\lambda,\frac{\rho^2}{\lambda}}}{\lambda}}$, we consider the first $K/2$ iterations where \begin{equation*}
    \eta_k = \eta = \min\crl*{\frac{1}{2(\bar{\xi}H+\lambda)},\frac{\lambda}{2\rho^2}}:
    \end{equation*}
\begin{align}
\E Q_\lambda(\bu_{K/2-1}) - Q_\lambda^* 
&\leq \prn*{\frac{1}{\eta} - \lambda + \eta\rho^2}\E\nrm*{\bu_{K/2-1} - u^*}^2 - \frac{1}{\eta}\E\nrm*{\bu_{K/2} - u^*}^2 + \frac{2\eta\prn*{\sigma^2 + \rho^2\nrm*{u^*}^2}}{M} \\
\implies \E\nrm*{\bu_{K/2} - u^*}^2
&\leq \prn*{1 - \eta\lambda + \eta^2\rho^2}\E\nrm*{\bu_{K/2-1} - u^*}^2 + \frac{2\eta^2\prn*{\sigma^2 + \rho^2\nrm*{u^*}^2}}{M} \\
&= \prn*{1 - \eta\lambda + \eta^2\rho^2}^{K/2}\E\nrm*{\bu_0 - u^*}^2 + \frac{2\eta^2\prn*{\sigma^2 + \rho^2\nrm*{u^*}^2}}{M}\sum_{k=0}^{K/2-1}\prn*{1 - \eta\lambda + \eta^2\rho^2}^k \\
&= \nrm*{u^*}^2\prn*{1 - \eta\lambda + \eta^2\rho^2}^{K/2} + \frac{2\eta^2\prn*{\sigma^2 + \rho^2\nrm*{u^*}^2}}{M}\frac{1 - \prn*{1 - \eta\lambda + \eta^2\rho^2}^{K/2}}{\eta\lambda - \eta^2\rho^2} \\
&\leq \nrm*{u^*}^2\prn*{1 - \eta\lambda + \eta^2\rho^2}^{K/2} + \frac{2\eta\prn*{\sigma^2 + \rho^2\nrm*{u^*}^2}}{M(\lambda - \eta\rho^2)} \\
&\leq \nrm*{u^*}^2\prn*{1 - \frac{1}{2}\eta\lambda}^{K/2} + \frac{4\eta\prn*{\sigma^2 + \rho^2\nrm*{u^*}^2}}{\lambda M}\ .\label{eq:sgd-halfway-iterate-distance}
\end{align}
This bounds the distance of the $(K/2)^{\textrm{th}}$ iterate to the optimum, from which we can upper bound the suboptimality of the averaged iterate. Let $W = \sum_{k=0}^K w_k = \sum_{k=K/2}^K w_k$. Then by the convexity of $Q_\lambda(u)$,
\begin{align}
\E &Q_\lambda\prn*{\frac{1}{W}\sum_{k=0}^K w_k \bu_k} - Q_\lambda^* \nonumber\\
&\leq \frac{1}{W}\sum_{k=0}^K\brk*{w_k\prn*{\frac{1}{\eta_k} - \lambda + \eta_k\rho^2}\E\nrm*{\bu_k - u^*}^2 - \frac{w_k}{\eta_k}\E\nrm*{\bu_{k+1} - u^*}^2 + \frac{2w_k\eta_k\prn*{\sigma^2 + \rho^2\nrm*{u^*}^2}}{M}} \\
&\leq \frac{w_{K/2}\prn*{\frac{1}{\eta_{K/2}} - \lambda + \eta_{K/2}\rho^2}\E\nrm*{\bu_{K/2} - u^*}^2}{W} \nonumber\\
&\qquad\qquad+ \frac{1}{W}\sum_{k=K/2+1}^K\brk*{\prn*{\frac{w_k}{\eta_k} - w_k\lambda + w_k\eta_k\rho^2 - \frac{w_{k-1}}{\eta_{k-1}}}\E\nrm*{\bu_k - u^*}^2 + \frac{2w_k\eta_k\prn*{\sigma^2 + \rho^2\nrm*{u^*}^2}}{M}}\ .
\end{align}
With our setting of $w_k = (a+k-K/2-1)$ and $\eta_k = \frac{4}{\lambda(a + k - K/2)}$ with $a = \frac{8}{\lambda}\max\crl*{\frac{\rho^2}{\lambda},\ \bar{\xi}H+\lambda}$, we have for $k > K/2$,
\begin{align}
\frac{\eta_{k-1}}{\eta_k} - \eta_{k-1}\lambda + \eta_{k-1}\eta_k\rho^2
&= \frac{a+k-K/2}{a+k-K/2-1} - \frac{4}{a+k-K/2-1} + \frac{16\rho^2}{\lambda^2(a+k-K/2)(a+k-K/2-1)} \\
&= \frac{a+k-K/2}{a+k-K/2-1} + \frac{1}{a+k-K/2-1}\prn*{2 - 4 + \frac{16\rho^2}{\lambda^2(a+k-K/2)}} \\
&\leq \frac{w_{k-1}}{w_k} + \frac{1}{a+k-K/2-1}\prn*{-2 + \frac{16\rho^2}{\lambda^2 a}} \\
&\leq \frac{w_{k-1}}{w_k}\ .
\end{align}
Therefore, we have
\begin{align}
\E Q_\lambda(\hat{x})& - Q_\lambda^* \nonumber\\
&\leq \frac{w_{K/2}\prn*{\frac{1}{\eta_{K/2}} - \lambda + \eta_{K/2}\rho^2}\E\nrm*{\bu_{K/2} - u^*}^2}{W} + \frac{2}{W}\sum_{k=K/2+1}^K\frac{w_k\eta_k\prn*{\sigma^2 + \rho^2\nrm*{u^*}^2}}{M} \\
&= \frac{(a-1)\prn*{\frac{\lambda a}{4} - \lambda + \frac{4\rho^2}{\lambda a}}\E\nrm*{\bu_{K/2} - u^*}^2}{W} + \frac{2(\sigma^2 + \rho^2\nrm*{u^*}^2)}{WM}\sum_{k=K/2+1}^K \frac{4(a+k-K/2 - 1)}{\lambda(a + k - K/2)} \\
&\leq \frac{\lambda a^2\E\nrm*{\bu_{K/2} - u^*}^2}{4W} + \frac{8K(\sigma^2 + \rho^2\nrm*{u^*}^2)}{\lambda WM} \\
&\leq \frac{\lambda a^2}{4W}\prn*{\nrm*{u^*}^2\prn*{1 - \frac{1}{2}\eta\lambda}^{K/2} + \frac{4\eta\prn*{\sigma^2 + \rho^2\nrm*{u^*}^2}}{\lambda M}} + \frac{8K(\sigma^2 + \rho^2\nrm*{u^*}^2)}{\lambda WM} \\
&= \frac{\lambda a^2\nrm*{u^*}^2}{4W}\prn*{1 - \frac{1}{2}\eta\lambda}^{K/2}  + \frac{(\sigma^2 + \rho^2\nrm*{u^*}^2)}{\lambda WM}\prn*{8K + \eta \lambda a^2} \\
&= \frac{\lambda a^2\nrm*{u^*}^2}{4W}\prn*{1 - \min\crl*{\frac{\lambda}{4(\bar{\xi}H+\lambda)},\ \frac{\lambda^2}{4\rho^2}}}^{K/2} + \frac{\prn*{8K + 4a}(\sigma^2 + \rho^2\nrm*{u^*}^2)}{\lambda WM}\ ,
\end{align}
where, for the third-to-last line we used \eqref{eq:sgd-halfway-iterate-distance}, and the last line we used that $\eta = \min\crl*{\frac{1}{2(\bar{\xi}H+\lambda)},\ \frac{\lambda}{2\rho^2}} = \frac{4}{\lambda a}$. Finally, we lower bound
\begin{equation}
W = \sum_{k=K/2}^K (a + k - K/2) = \sum_{i=0}^{K/2} a + i \geq \frac{1}{2}aK + \frac{1}{8}K^2.
\end{equation}
Thus, since $K \geq \frac{2}{\lambda}\max\crl*{\frac{\rho^2}{\lambda},\ \bar{\xi}H+\lambda} = \frac{a}{4}$, we conclude
\begin{align}
\E Q_\lambda(\hat{x}) - Q_\lambda^* 
&\leq 96\lambda\nrm*{u^*}^2\prn*{1 - \min\crl*{\frac{\lambda}{4(\bar{\xi}H+\lambda)},\ \frac{\lambda^2}{4\rho^2}}}^{K/2} + \frac{96(\sigma^2 + \rho^2\nrm*{u^*}^2)}{\lambda MK}\ .
\end{align}
Combining this and \eqref{eq:sgd-with-multiplicative-noise-case1} completes the proof.
\end{proof}

\section{Proof of \Cref{thm:combined-alg}}\label{app:combined-alg}

\combinedalg*

\begin{proof}
We first recall the hyperparameters from \Cref{tab:hyperparams}, whose settings we will refer to throughout the course of the proof:
\begin{table}[H]
\def\arraystretch{1.5}%
\centering
\begin{tabular}{ l l } 
 \toprule
 Hyperparameter Setting  & Description\\
 \midrule 
$T \defeq \Bigg\lfloor\frac{R}{4\zeta}\log^2\prn*{\prn*{\frac{R}{\zeta}}}\Bigg\rfloor$ (for $\zeta = 4096 + 4\prn*{80 + 32\log K + 24\log(1+2\alpha B)}^2$) & Main iterations\\[0.3em]
$\beta := 0$ & Momentum\\
 $\rr \defeq \min\crl*{\frac{32B}{T}\log(TK),\,\frac{1}{5\alpha}}$ & Trust-region radius\\
 $\bar{\xi} \defeq \exp(\alpha \rr)$ & Local stability\\
 $\lmin \defeq \max\crl*{\frac{2eH}{K-2},\,\frac{2\rho}{\sqrt{K}},\, \frac{32eH\log(51200)}{K},\, \frac{4\rho\sqrt{2\log(51200)}}{\sqrt{K}},\,\frac{320\sqrt{2}\rho}{\sqrt{MK}},\,\frac{320\sigma}{\rr\sqrt{MK}},\, \frac{8eH}{K-16}}$ & Regularization bound\\
 $N \defeq \left\lceil1 + \frac{5}{2}\log\frac{H(B + 5T\rr)}{3\lmin\rr}\right\rceil$ & Binary search iterations\\
 $C \defeq \left\lceil 8\log\prn*{\ceil{\log_2 N}\prn*{4 + \frac{e H}{\lmin} + \frac{80H(B + 5T\rr)}{\lmin \rr}}} \right\rceil$ & Reg.\ quadratic repetitions\\[0.8em]
 \toprule 
\end{tabular}
\end{table}

We also note that since all of the updates $\updatet_t$ have norm at most $5\rr$, $\nrm{x_t - x^*} \leq B + 5T\rr$ for all $t$, and therefore by the $H$-smoothness of $F$, $\nrm{\nabla F(x_t)} \leq H(B + 5T\rr)$ for all $t$. Furthermore, since $F$ is $H$-smooth and $\rr \leq \frac{1}{5\alpha}$, $\bar{\xi} \nabla^2 F(x_t) \preceq e H \boldI$ for all $t$. Therefore, our settings of $N$ and $C$ satisfy the conditions of \Cref{lem:constrained-quadratic-correctness}, and for each $t$,
\begin{equation}\label{eq:alg2-guarantee1}
\E Q_t(\updatet_t) - \min_{\Delta x:\nrm{\Delta x}\leq\frac{1}{2}\rr} Q_t(\Delta x) \leq \frac{3}{4}\prn*{Q_t(0) - \min_{\Delta x:\nrm{\Delta x}\leq\frac{1}{2}\rr} Q_t(\Delta x)} + \epsilon(\lambda_{4\rr}) + \frac{\lmin \rr^2}{4}
\end{equation}
as long as the error guarantee of \Cref{alg:one-shot-averaging} satisfies for all $\lambda \geq \lmin$
\begin{equation}
\E Q_\lambda(\hat{u}) - \min_u Q_\lambda(u) \leq \epsilon(\lambda) = \frac{\lambda(r^*(\lambda)^2 + \rr^2)}{800}\ .
\end{equation}
By \Cref{lem:sgd-with-multiplicative-noise}, since the objectives are such that $\nrm*{\bar{\xi} \nabla^2 F(x_t) + \lambda \boldI}_2 \leq e H +\lambda$ for all $t$ and $\lambda$, the output with optimally chosen stepsizes have error at most
\begin{align}
&\E  Q_\lambda(\hat{u}) - \min_u Q_\lambda(u)
\leq \tilde{\epsilon}(\lambda) \\
&\defeq 
\begin{cases}
2\max\crl*{eH+\lambda, \frac{\rho^2}{\lambda}}r^*(\lambda)^2 \exp\prn*{-\frac{K+1}{4}\min\crl*{\frac{\lambda}{eH + \lambda},\ \frac{\lambda^2}{\rho^2}}} + \frac{2(\sigma^2 + \rho^2r^*(\lambda)^2)}{\lambda MK} & K \leq \frac{2}{\lambda}\max\crl*{eH+\lambda,\, \frac{\rho^2}{\lambda}} \\
96\lambda r^*(\lambda)^2\exp\prn*{ -\frac{K}{8}\min\crl*{\frac{\lambda}{eH+\lambda},\ \frac{\lambda^2}{\rho^2}}} + \frac{96(\sigma^2 + \rho^2r^*(\lambda)^2)}{\lambda MK} & K > \frac{2}{\lambda}\max\crl*{eH+\lambda,\, \frac{\rho^2}{\lambda}}\ .
\end{cases}
\end{align}
With our choice of
\begin{equation}
\lmin = \max\crl*{\frac{2eH}{K-2},\,\frac{2\rho}{\sqrt{K}},\, \frac{32eH\log(51200)}{K},\, \frac{4\rho\sqrt{2\log(51200)}}{\sqrt{K}},\,\frac{320\sqrt{2}\rho}{\sqrt{MK}},\,\frac{320\sigma}{\rr\sqrt{MK}},\, \frac{8eH}{K-16}}\ ,
\end{equation}
we note that
\begin{equation}
K \geq \frac{2}{\lambda}\max\crl*{eH + \lambda,\,\frac{\rho^2}{\lambda}}\ ,
\end{equation}
so for $\lambda \geq \lmin$
\begin{equation}
\tilde{\epsilon}(\lambda) \leq 96\lambda r^*(\lambda)^2\exp\prn*{ -\frac{K}{8}\min\crl*{\frac{\lambda}{eH+\lambda},\ \frac{\lambda^2}{\rho^2}}} + \frac{96(\sigma^2 + \rho^2r^*(\lambda)^2)}{\lambda MK}\ .
\end{equation}
Furthermore, $K \geq 175$ and $\lambda\geq\lmin\geq \max\crl*{\frac{32eH\log(51200)}{K},\, \frac{4\rho\sqrt{2\log(51200)}}{\sqrt{K}}}$ implies 
\begin{equation}
96\exp\prn*{-\frac{K}{8}\min\crl*{\frac{\lambda}{eH+\lambda},\ \frac{\lambda^2}{\rho^2}}} \leq \frac{1}{1600}\ .
\end{equation}
Likewise, $\lambda\geq\lmin\geq \frac{320\sqrt{2}\rho}{\sqrt{MK}}$ implies 
\begin{equation}
\frac{96\rho^2}{\lambda^2 MK} \leq \frac{1}{1600}\ .
\end{equation}
Finally, $\lambda\geq\lmin\geq \frac{320\sigma}{\rr\sqrt{MK}}$ implies
\begin{equation}
\frac{96\sigma^2}{\lambda^2 MK} \leq \frac{\rr^2}{800}\ .
\end{equation}
Putting these together, we conclude that for $\lambda \geq \lmin$
\begin{equation}
\tilde{\epsilon}(\lambda) \leq \frac{\lambda(r^*(\lambda)^2 + \rr^2)}{800} = \epsilon(\lambda).
\end{equation}

Combining this with \eqref{eq:alg2-guarantee1}, we conclude that the output of \Cref{alg:constrained-quadratic} satisfies
\begin{align}
\E& Q_t(\updatet_t) - \min_{\Delta x:\nrm{\Delta x}\leq\frac{1}{2}\rr} Q_t(\Delta x) \nonumber\\
&\leq \frac{3}{4}\prn*{Q_t(0) - \min_{\Delta x:\nrm{\Delta x}\leq\frac{1}{2}\rr} Q_t(\Delta x)} \nonumber\\
&\qquad+ 512\lambda \rr^2 \exp\prn*{ -\frac{K}{8}\min\crl*{\frac{\lambda}{eH+\lambda},\ \frac{\lambda^2}{\rho^2}}} + \frac{96(\sigma^2 + 16\rho^2\rr^2)}{\lambda MK} + \frac{\lmin \rr^2}{4} \\
&\leq \frac{3}{4}\prn*{Q_t(0) - \min_{\Delta x:\nrm{\Delta x}\leq\frac{1}{2}\rr} Q_t(\Delta x)} \nonumber\\
&\qquad+ 512\lambda \rr^2 \exp\prn*{ -\frac{K}{8}\min\crl*{\frac{\lambda}{eH+\lambda},\ \frac{\lambda^2}{\rho^2}}} + \frac{\sigma\rr + 16\rho\rr^2}{2\sqrt{MK}} + \frac{\lmin \rr^2}{4}\ .
\end{align}

Now, we upper bound $\lambda \mapsto \lambda \exp\prn*{ -\frac{K}{8}\min\crl*{\frac{\lambda}{eH+\lambda},\ \frac{\lambda^2}{\rho^2}}}$ for $\lambda \in \Lambda_1$. First, 
\begin{equation}
\lambda \exp\prn*{ -\frac{K}{8}\min\crl*{\frac{\lambda}{eH+\lambda},\ \frac{\lambda^2}{\rho^2}}} = \max\crl*{\lambda \exp\prn*{ -\frac{\lambda K}{8eH+8\lambda}},\, \lambda\exp\prn*{ -\frac{\lambda^2 K}{8\rho^2}}}\ .
\end{equation}
Considering each term separately, 
\begin{equation}
\frac{d}{d\lambda}\brk*{\lambda \exp\prn*{ -\frac{\lambda K}{8eH+8\lambda}}} = \exp\prn*{-\frac{\lambda K}{8eH+8\lambda}}\prn*{1 - \frac{eHK\lambda}{8(eH+\lambda)^2}}\ .
\end{equation}
This is less than zero if $8(eH+\lambda)^2 \leq eHK\lambda$, i.e.,
\begin{equation}
\frac{eH}{16}\prn*{K-16 - \sqrt{(K-16)^2 - 16}} \leq \lambda \leq \frac{eH}{16}\prn*{K-16 + \sqrt{(K-16)^2 - 16}}.
\end{equation}
With our choice of $\lmin \geq \frac{8eH}{K-16}$, for any $\lambda \geq \lmin$ and $K \geq 175$,
\begin{equation}
\lambda \geq \frac{8eH}{K-16} \geq \frac{eH}{16}\prn*{K-16 - \sqrt{(K-16)^2 - 16}},
\end{equation}
so the left side of this inequality is satisfied. Thus, for $\lambda \in \Lambda_1$ such that $\lambda \leq \frac{eH}{16}\prn*{K-16 + \sqrt{(K-16)^2 - 16}}$, 
\begin{equation}
\lambda \exp\prn*{ -\frac{\lambda K}{8eH+8\lambda}} \leq \lmin \exp\prn*{ -\frac{\lmin K}{8eH+8\lmin}}\ .
\end{equation}
Also, if $\lambda > \frac{eH}{16}\prn*{K-16 + \sqrt{(K-16)^2 - 16}}$, then since $K \geq 175$
\begin{equation}
\lambda \exp\prn*{ -\frac{\lambda K}{8eH+8\lambda}} \leq \lambda\exp\prn*{ -\frac{\frac{eH(K-16)}{16}K}{8eH+\frac{eH(K-16)}{2}}} \leq \lambda\exp\prn*{ -\frac{K}{10}}\ .
\end{equation}
Furthermore, for $\lambda \in \Lambda_1$, 
\begin{equation}
\lambda \leq \lmin\prn*{\frac{3}{2}}^{N-1} \leq \frac{3\lmin}{2}\frac{H(B + 5T\rr)}{3\lmin \rr} = \frac{3H(B + 5T\rr)}{6 \rr}\ .
\end{equation}
Therefore, for any $\lambda \in \Lambda_1$ 
\begin{equation}
\lambda \exp\prn*{ -\frac{\lambda K}{8eH+8\lambda}} \leq \max\crl*{\lmin \exp\prn*{ -\frac{\lmin K}{8eH+8\lmin}},\,\frac{3H(B + 5T\rr)}{6 \rr}\exp\prn*{ -\frac{K}{10}}}\ .
\end{equation}

Similarly, 
\begin{equation}
\frac{d}{d\lambda}\brk*{\lambda \exp\prn*{ -\frac{\lambda^2 K}{8\rho^2}}} = \exp\prn*{ -\frac{\lambda^2 K}{8\rho^2}}\prn*{1 - \frac{\lambda^2 K}{4\rho^2}}\ .
\end{equation}
This is negative for all $\lambda \geq \lmin \geq \frac{2\rho}{\sqrt{K}}$, so
\begin{equation}
\lambda \exp\prn*{ -\frac{\lambda^2 K}{8\rho^2}} \leq \lmin \exp\prn*{ -\frac{\lmin^2 K}{8\rho^2}}\ .
\end{equation}
We conclude that
\begin{align}
\E& Q_t(\updatet_t) - \min_{\Delta x:\nrm{\Delta x}\leq\frac{1}{2}\rr} Q_t(\Delta x) \nonumber\\
&\leq \frac{3}{4}\prn*{Q_t(0) - \min_{\Delta x:\nrm{\Delta x}\leq\frac{1}{2}\rr} Q_t(\Delta x)} + \frac{\sigma\rr + 16\rho\rr^2}{2\sqrt{MK}} + \frac{\lmin \rr^2}{4} \nonumber\\
&\qquad+ 512\rr^2\max\crl*{\lmin \exp\prn*{ -\frac{\lmin K}{8eH+8\lmin}},\,\frac{3H(B + 5T\rr)}{6 \rr}\exp\prn*{ -\frac{K}{10}},\,\lmin \exp\prn*{ -\frac{\lmin^2 K}{8\rho^2}}}\ .
\end{align}

Now, because $F$ is $\alpha$-quasi-self-concordant, by \Cref{lem:qsc-sensitivity}, $F$ is $\exp(\alpha r)$-locally stable, so with our choice of $\rr \leq \frac{1}{5\alpha}$, we have that $\exp(\frac{1}{2}\rr) \exp(5\rr) \leq e^{1.1} \leq 4$. Thus, it follows from \Cref{lem:qsc-newton}, for $\theta = \frac{3}{4}$, combined with the guarantee on the output of \Cref{alg:constrained-quadratic} from \Cref{lem:constrained-quadratic-correctness}, that
\begin{align}
&\E F(x_T) - F^* \nonumber\\
&\leq \E\brk*{F(x_0) - F^*}\exp\prn*{-\frac{T\rr}{32B}} + \frac{32B}{\rr}\prn*{\frac{\sigma\rr + 16\rho\rr^2}{2\sqrt{MK}} + \frac{\lmin \rr^2}{4}} \nonumber\\
&\qquad+ \frac{32B}{\rr} \cdot 512\rr^2\max\crl*{\lmin \exp\prn*{ -\frac{\lmin K}{8eH+8\lmin}},\,\frac{3H(B + 5T\rr)}{6 \rr}\exp\prn*{ -\frac{K}{10}},\,\lmin \exp\prn*{ -\frac{\lmin^2 K}{8\rho^2}}} \\
&\leq \E\brk*{F(x_0) - F^*}\exp\prn*{-\frac{T\rr}{32B}} + \frac{32\sigma B + 512\rho B\rr}{2\sqrt{MK}} + 8\lmin B \rr \nonumber\\
&\qquad+ 2^{14}\max\crl*{\lmin B\rr \exp\prn*{ -\frac{\lmin K}{8eH+8\lmin}},\,\frac{3H(B^2 + 5TB\rr)}{6}\exp\prn*{ -\frac{K}{10}},\,\lmin B\rr \exp\prn*{ -\frac{\lmin^2 K}{8\rho^2}}} \\
&\leq \E\brk*{F(x_0) - F^*}\exp\prn*{-\frac{T\rr}{32B}} + \frac{32\sigma B + 512\rho B\rr}{2\sqrt{MK}} \nonumber\\
&\qquad+ \lmin B \rr \prn*{8 + 2^{14}\max\crl*{\exp\prn*{ -\frac{\lmin K}{8eH+8\lmin}},\,\frac{3H(B + 5T\rr)}{6\lmin \rr}\exp\prn*{ -\frac{K}{10}},\,\exp\prn*{ -\frac{\lmin^2 K}{8\rho^2}}}}\ .
\end{align}

We have, for a constant $c$,
\begin{align}
\lmin 
&= \max\crl*{\frac{2eH}{K-2},\,\frac{2\rho}{\sqrt{K}},\, \frac{32eH\log(51200)}{K},\, \frac{4\rho\sqrt{2\log(51200)}}{\sqrt{K}},\,\frac{320\sqrt{2}\rho}{\sqrt{MK}},\,\frac{320\sigma}{\rr\sqrt{MK}},\, \frac{8eH}{K-16}} \\
&= \max\crl*{\frac{32eH\log(51200)}{K},\, \frac{4\rho\sqrt{2\log(51200)}}{\sqrt{K}},\,\frac{320\sqrt{2}\rho}{\sqrt{MK}},\,\frac{320\sigma}{\rr\sqrt{MK}}} \\
&= c\cdot\max\crl*{\frac{H}{K},\, \frac{\rho}{\sqrt{K}},\,\frac{\sigma}{\rr\sqrt{MK}}}\ .
\end{align}

So, for a constant $c'$, and using $\E F(x_0) - F^* \leq \frac{HB^2}{2}$,
\begin{align}
&\E F(x_T) - F^* \nonumber\\
&\leq c'\cdot\bigg(HB^2\exp\prn*{-\frac{T\rr}{32B}} + \frac{\sigma B + \rho B\rr}{\sqrt{MK}} \nonumber\\
&\qquad+ \lmin B \rr \prn*{1 + \max\crl*{\exp\prn*{ -\frac{\lmin K}{H+\lmin}},\,\frac{H(B + T\rr)}{\lmin \rr}\exp\prn*{ -\frac{K}{10}},\,\exp\prn*{ -\frac{\lmin^2 K}{\rho^2}}}}\bigg) \\
&\leq c'\cdot\prn*{HB^2\exp\prn*{-\frac{T\rr}{32B}} + \frac{\sigma B + \rho B\rr}{\sqrt{MK}} + \max\crl*{\lmin B \rr,\,H(B^2 + TB\rr)\exp\prn*{ -\frac{K}{10}}}} \\
&= c'\cdot\bigg(HB^2\exp\prn*{-\frac{T\rr}{32B}} + \frac{\sigma B + \rho B\rr}{\sqrt{MK}} \nonumber\\
&\qquad+ \max\crl*{\frac{HB\rr}{K},\, \frac{\rho B\rr}{\sqrt{K}},\,\frac{\sigma B}{\sqrt{MK}},\,H(B^2 + TB\rr)\exp\prn*{ -\frac{K}{10}}}\bigg) \\
&= c'\cdot \prn*{HB^2\exp\prn*{-\frac{T\rr}{32B}} + \frac{\sigma B}{\sqrt{MK}} + \frac{HB\rr}{K} + \frac{\rho B\rr}{\sqrt{K}} + H(B^2 + TB\rr)\exp\prn*{ -\frac{K}{10}}}.
\end{align}
So, since $\rr = \min\crl*{\frac{32B}{T}\log(TK),\,\frac{1}{5\alpha}}$, and using the fact that, for $\zeta, a, b > 0$, $e^{-\zeta\min\crl*{a,b}} \leq e^{-\zeta a} + e^{-\zeta b}$, we have, for a constant $c''$,
\begin{align}
\E F(x_T) - F^* 
&\leq c''\cdot\bigg(HB^2\exp\prn*{-\frac{T}{160\alpha B}} + \frac{HB^2}{TK} + \frac{\sigma B}{\sqrt{MK}} + \frac{HB^2\log TK}{TK} \nonumber\\
&\qquad\qquad+ \frac{\rho B^2\log TK}{T\sqrt{K}} + HB^2\log(TK)\exp\prn*{ -\frac{K}{10}} \bigg)\\
&\leq c''\cdot\bigg(HB^2\exp\prn*{-\frac{T}{160\alpha B}} + \frac{\sigma B}{\sqrt{MK}} + \frac{HB^2\log TK}{TK} \nonumber\\
&\qquad\qquad+ \frac{\rho B^2\log TK}{T\sqrt{K}} + HB^2\log(TK)\exp\prn*{ -\frac{K}{10}} \bigg)\ ,
\end{align}
where the last inequality follows from the fact $\log(TK) \geq 1$, since $TK \geq 175$.

Finally, each call to \Cref{alg:constrained-quadratic} requires at most $C \left\lceil \log N \right\rceil$ rounds of communication (one for each call to Algorithm \ref{alg:one-shot-averaging}). Therefore, we can implement up to $R / (C\left\lceil \log N \right\rceil)$ iterations of Algorithm \ref{alg:stochastic-newton} using our $R$ rounds of communication. We recall that $N$ and $C$ are set as
\begin{align}
N &= \left\lceil1 + \frac{5}{2}\log\frac{H(B + 5T\rr)}{3\lmin\rr}\right\rceil \\
C &= \left\lceil 8\log\prn*{\ceil{\log_2 N}\prn*{4 + \frac{e H}{\lmin} + \frac{80H(B + 5T\rr)}{\lmin \rr}}} \right\rceil\ .
\end{align}
Therefore, we need to choose $T$ such that $T \leq \frac{R}{C\left\lceil\log N\right\rceil}$. To provide an explicit lower bound on how large $T$ can be, we therefore lower bound the right hand side. 
First, we have
\begin{align}
N 
&= \left\lceil1 + \frac{5}{2}\log\frac{H(B + 5T\rr)}{3\lmin\rr}\right\rceil \\
&\leq 2 + \frac{5}{2}\log\frac{H\prn*{B + 5T\frac{B}{T}\log(TK)}}{3\frac{2eH}{K-2}\min\crl*{\frac{32B}{T}\log(TK),\,\frac{1}{5\alpha}}} \\
&\leq 2 + \frac{5}{2}\log\frac{BK\log(TK)}{e\min\crl*{\frac{32B}{T}\log(TK),\,\frac{1}{5\alpha}}} \\
&\leq 2 + \frac{5}{2}\max\crl*{\log (TK\log(TK)),\, \log(2\alpha BK\log(TK))} \\
&\leq 2 + 5\log(1 + 2\alpha B) + 5\log (TK).
\end{align}
Similarly, 
\begin{align}
C 
&= \ceil*{8\log\prn*{\ceil{\log_2 N}\prn*{4 + \frac{e H}{\lmin} + \frac{80H(B + 5T\rr)}{\lmin \rr}}}} \\
&\leq 1 + 8\log\prn*{\ceil{\log_2 N}\prn*{4 + \frac{e H}{\frac{2eH}{K-2}} + 240\max\crl*{TK\log(TK),\, 2\alpha BK\log(TK)}}} \\
&\leq 1 + 8\log\prn*{\ceil{\log_2 N}\prn*{4K + 240\max\crl*{T^2K^2,\, 2\alpha BTK^2}}} \\
&\leq 1 + 8\log\prn*{\prn*{1 + \log N}\prn*{353T^2K^2 + 693\alpha BTK^2}} \\
&\leq 1 + 8\log(693) + 16\log(TK) + 8\log(1 + \alpha B) + 8\log(1 + \log(N)) \\
&\leq 54 + 16\log(TK) + 8\log(1 + \alpha B) + 16\log\log N \\
&\leq 80 + 32\log(TK) + 24\log(1+2\alpha B).
\end{align}
Therefore, 
\begin{align}
\frac{R}{C\left\lceil \log N \right\rceil} 
&\geq \frac{R}{\prn*{80 + 32\log(TK) + 24\log(1+2\alpha B)}\left\lceil\log\prn*{2 + 5\log(1 + 2\alpha B) + 5\log (TK)}\right\rceil} \\
&\geq \frac{R}{\prn*{80 + 32\log(TK) + 24\log(1+2\alpha B)}^2} \\
&\geq \frac{R}{2048 + 2\prn*{80 + 32\log K + 24\log(1+2\alpha B)}^2}\ .
\end{align}
So, it suffices to choose $T$ such that
\begin{equation}
T\log^2(T) \leq \frac{R}{2048 + 2\prn*{80 + 32\log K + 24\log(1+2\alpha B)}^2}\ .
\end{equation}

Note that equality holds for
\begin{align*}
    T &= \exp\prn*{2W\prn*{\prn*{\frac{R}{4096 + 4\prn*{80 + 32\log K + 24\log(1+2\alpha B)}^2}}^{1/2}}},
\end{align*}
where $W(\cdot)$ denotes the Lambert $W$ function. Thus, because $W(x) \geq \log(x)-\log\log(x)$ for $x \geq e$, and since we assume \[R \geq \tilde{\Omega}(1) = e^2\prn*{4096 + 4\prn*{80 + 32\log K + 24\log(1+2\alpha B)}^2},\] we have that
\begin{align*}
    &\exp\prn*{2W\prn*{\prn*{\frac{R}{4096 + 4\prn*{80 + 32\log K + 24\log(1+2\alpha B)}^2}}}^{1/2}}\\
    &\geq \exp\Bigg(2\prn*{\log\prn*{\frac{R}{4096 + 4\prn*{80 + 32\log K + 24\log(1+2\alpha B)}^2}}^{1/2}} \\
    &\qquad - \log\log\prn*{\frac{R}{4096 + 2\prn*{80 + 32\log K + 24\log(1+2\alpha B)}^2}}^{1/2}\Bigg)\\
    & = \frac{1}{4}\prn*{\frac{R}{4096 + 4\prn*{80 + 32\log K + 24\log(1+2\alpha B)}^2}}\\
    &\qquad\cdot\log^2\prn*{\prn*{\frac{R}{4096 + 4\prn*{80 + 32\log K + 24\log(1+2\alpha B)}^2}}}.
\end{align*}

Thus, letting 
\begin{align*}
    T &= \Bigg\lfloor\frac{1}{4}\prn*{\frac{R}{4096 + 4\prn*{80 + 32\log K + 24\log(1+2\alpha B)}^2}}\\
    &\qquad\cdot\log^2\prn*{\prn*{\frac{R}{4096 + 4\prn*{80 + 32\log K + 24\log(1+2\alpha B)}^2}}}\Bigg\rfloor,
\end{align*}
and using $\tilde{O}$ notation to hide polylogarithmic factors in $R$, $K$, and $\alpha B$, we have
\begin{align}
\E [F(x_T)] - F^* \leq HB^2\prn*{\exp\prn*{-\frac{R}{\tilde{O}\prn*{\alpha B}}} + \exp\prn*{ -\frac{K}{O(1)}}} + \tilde{O}\prn*{\frac{\sigma B}{\sqrt{MK}} + \frac{HB^2}{KR} + \frac{\rho B^2}{\sqrt{K}R}}\ .
\end{align}

This completes the proof.
\end{proof}

\section{Additional Information for Experiments}\label{sec:add_exp}
\subsection{Baselines}\label{sec:baselines}
We have compared \fedsnlite{} against the two variants of \fedac{} \citep{yuan2020federated}, Minibatch SGD \citep{dekel2012optimal}, and Local SGD \citep{zinkevich2010parallelized}. Two settings of hyperparameters are considered for \fedac{} in \cite{yuan2020federated} for strongly convex functions:
\begin{itemize}
    \item \textbf{\textsc{FedAc-I}}: $\eta \in (0,1/H],\ \gamma = \max\left\{\sqrt{\frac{\eta}{\lambda K}}, \eta\right\},\ \alpha = \frac{1}{\gamma \lambda}, \beta = \alpha + 1$;
    \item \textbf{\textsc{FedAc-II}}: $\eta \in (0,1/H],\ \gamma = \max\left\{\sqrt{\frac{\eta}{\lambda K}}, \eta\right\},\ \alpha = \frac{3}{2\gamma \lambda} - \frac{1}{2}, \beta = \frac{2\alpha^2-1}{\alpha -1}$,
\end{itemize}
where $H$ is the smoothness constant as in Assumption \ref{assump:a1}, $\lambda$ is an estimate of the strong convexity, and $\eta$ is the learning rate which has to be tuned. Thus, a limitation of \fedac{} is that it requires either the knowledge of $\lambda$ (say, through explicit $\lambda$-regularization), or that the algorithm adds regularization to the objective itself (this is how \cite{yuan2020federated} present \fedac{} for general convex functions). In our experiments we have both of these settings, i.e., \fedac{} with internal regularization $\lambda$ or explicitly $\lambda$-regularized objectives. For brevity, in \Cref{alg:fedac} we present \fedac{} with five hyperparameters: $\alpha$, $\beta$, $\eta$, $\gamma$, and $\lambda$. When the objective is regularized we use $\lambda=0$ (c.f. \Cref{subsec:experiment2}, Experiment 2), whereas otherwise we tune $\lambda$ (c.f. \Cref{subsec:experiment1}, Experiment 1), to ensure the best possible performance for \fedac{}.

\begin{algorithm}[H]
   \caption{\fedac{}$(x_0, \alpha, \beta, \eta, \gamma, \lambda)$}
   \label{alg:fedac}
\begin{algorithmic}
    \STATE \hspace{-1em}(Operating on objective $F(\cdot) + \lambda/2\norm{.}^2$ as opposed to $F(.)$ with stochastic gradient oracle $g_{\lambda}(\cdot;\cdot)$\footnotemark.)
    \STATE \textbf{Intitialize:} $x^{ag,m}_0 = x^m_0 = x_0$ for all $m\in[M]$
   \FOR{$t = 0,1,\dots,T-1$}
   \FOR{every worker $m\in[M]$ \textbf{in parallel}}
   \STATE $x^{md, m}_t \gets \beta^{-1}x^m_t + (1-\beta^{-1})x^{ag, m}_t$
   \STATE $g^m_t \gets g_{\lambda}(x^{md, m}_t, z^m_t)$ \hfill{} $\triangleright$ Query the stochastic first-order oracle at $x^{md, m}_t$, for $z^m_t\sim \mathcal{D}$
   \STATE $v^{ag,m}_{t+1} \gets x^{md, m}_t - \eta \cdot g^m_t$
   \STATE $v^{m}_{t+1} \gets (1-\alpha^{-1})x^m_t + \alpha^{-1}x^{md, m}_t - \gamma g^m_t$
   \IF{$t$ mod $K = -1$}
   \STATE $x^m_{t+1} \gets \frac{1}{M}\sum_{m'=1}^{M}v^{m'}_{t+1}$
   \STATE $x^{ag,m}_{t+1} \gets \frac{1}{M}\sum_{m'=1}^{M}v^{ag,m'}_{t+1}$
   \ELSE
   \STATE $x^m_{t+1}\gets v^m_{t+1}$
   \STATE $x^{ag,m}_{t+1}\gets v^{ag,m}_{t+1}$
   \ENDIF
   \ENDFOR
   \ENDFOR
   \STATE \textbf{Return:} $\bar{x}^{ag}_T = \frac{1}{M}\sum_{m'=1}^{M}x^{ag, m'}_T$
\end{algorithmic}

\end{algorithm}
\footnotetext{Note that $g_\lambda(x;z) = g(x;z) + \lambda x$, i.e., the stochastic oracle for the regularized objective can always be obtained using the oracle for the unregularized objective (see Assumption \ref{assump:a1}).)}
In \Cref{alg:lsgd} we describe Local SGD (a.k.a. \textsc{FedAvg}) \citep{zinkevich2010parallelized} with learning rate $\eta$ and Polyak's momentum (a.k.a. heavy ball method) parameter $\beta$. Setting $\beta=0$ recovers the familiar algorithm as analyzed in \cite{woodworth2020local}. Finally, in \Cref{alg:mbsgd} we describe Minibatch SGD with fixed learning rate $\eta$ and momentum parameter $\beta$. Note that in our experiments we compare the algorithms using the same number of machines $M$, communication rounds $R$ and theoretical parallel runtime $T$ against each other, where $T=KR$. Also note that unlike \Cref{alg:fedac}, there is no internal regularization in \Cref{alg:lsgd,alg:mbsgd}. While conducting Experiment 2 (c.f., \cref{subsec:experiment2}) we assume that $F(.)$ is regularized, so that \Cref{alg:lsgd,alg:mbsgd} instead minimize $F(.) + \frac{\lambda}{2}\norm{.}^2$ and have access to $g_\lambda(.;.)$. The reason why $\lambda$ is not presented as a hyperparameter for any algorithms beside \fedac{}, is because the regularization is a part of the algorithm \fedac{} itself which we tune besides its learning rate (c.f., Section E in \cite{yuan2020federated}).  

\begin{algorithm}[H]
   \caption{Local SGD$(x_0, \beta, \eta)$}
   \label{alg:lsgd}
\begin{algorithmic}
\STATE (Operating on objective $F(\cdot)$ with stochastic gradient oracle $g(\cdot;\cdot)$.)
\STATE \textbf{Intitialize:} $x_{0,0}^{m}\gets x_{0}$ for all $m\in [M]$
   \FOR{$r = 0,\dots,R-1$}
   \FOR{every worker $m\in[M]$ \textbf{in parallel}}
   \FOR{$k = 0, \dots, K-1$}
        \STATE $g^m_{r,k} \gets g(x^m_{r,k}; z_{r,k}^m)$ \hfill{} $\triangleright$ Query the stochastic first-order oracle at $x^m_{r,k}$ (for $z^m_{r,k}\sim \mathcal{D}$ drawn by the oracle)
        \STATE $x^m_{r,k+1} \gets x^m_{r,k} - \eta g_{r,k}^m + \mathbbold{1}_{k>0}\beta(x^m_{r,k} - x^m_{r,k-1})$
    \ENDFOR
   \ENDFOR
   \STATE $x_{r+1} \gets \frac{1}{M}\sum_{m'=1}^{M}x_{r,K}^{m'}$
   \ENDFOR
   \STATE \textbf{Return:} $x_R$
\end{algorithmic}
\end{algorithm}

\begin{algorithm}[H]
   \caption{Minibatch SGD$(x_0, \beta, \eta)$}
   \label{alg:mbsgd}
\begin{algorithmic}
    \STATE (Operating on objective $F(\cdot)$ with stochastic gradient oracle $g(\cdot;\cdot)$.)
    \STATE \textbf{Intitialize:} $x^m_0 = x_0$ for all $m\in[M]$
   \FOR{$r = 0,\dots,R-1$}
   \FOR{every worker $m\in[M]$ \textbf{in parallel}}
   \FOR{$k = 0, \dots, K-1$}
        \STATE $g^m_{r,k} \gets g(x_r; z_{r,k}^m)$ \hfill{} $\triangleright$ Query the stochastic first-order oracle at $x_r$ (for $z^m_{r,k}\sim \mathcal{D}$ drawn by the oracle)
    \ENDFOR
   \ENDFOR
   \STATE $g_r \gets \frac{1}{KM}\sum_{k'=0}^{K-1}\sum_{m'=1}^{M}g^{m'}_{r,k'}$
   \STATE $x_{r+1} \gets x_{r} - \eta g_{r} + \mathbbold{1}_{r>0}\beta (x_{r} - x_{r-1})$
   \ENDFOR
   \STATE \textbf{Return:} $x_R$
\end{algorithmic}
\end{algorithm}

In addition, we may see that \ossgd{} with $\eta_k(\lambda) = \eta$ and $w_k(\lambda) = \frac{1}{K}$ (as in each iteration of \fedsnlite{}) is a special case of Local SGD (\Cref{alg:lsgd}) with $R=1$. Thus to summarize, in our experiments, we have compared the following algorithms: 
\begin{itemize}
    \item \fedsnlite{} (which uses \ossgd{}) without momentum i.e., $\beta = 0$ or with optimally tuned momentum $\beta\in\{0.1, 0.3, 0.5, 0.7, 0.9\}$,
    \item \textsc{FedAc-1} and \textsc{FedAC-2}, with either no internal regularization i.e., $\lambda = 0$ (in experiment 2 when the objective is explicitly regularized) or optimally tuned internal regularization $\lambda\in \{$1e-~2, 1e-3, 1e-4, 1e-5, 1e-6$\}$ (in experiment 1),
    \item Local SGD without momentum i.e., $\beta = 0$ or with tuned momentum $\beta\in\{0.1, 0.3, 0.5, 0.7, 0.9\}$,
    \item Minibatch SGD without momentum i.e., $\beta = 0$ or with tuned momentum $\beta\in\{0.1, 0.3, 0.5, 0.7, 0.9\}$,
\end{itemize}

We search for the best learning rate $\eta \in \{0.0001,0.0002,$ $0.0005, 0.001, 0.002, 0.005, 0.01, 0.02, 0.05, 0.1, 0.2, \\0.5, 1, 2, 5, 10, 20\}$ for every configuration of each algorithm. We verify (along with \cite{yuan2020federated}) that the optimal learning rate always lies in this range.

\subsection{More implementation details}\label{sec:implementation_details}

\textbf{Dataset.} Following the setup in \cite{yuan2020federated}, we use the full \href{https://www.csie.ntu.edu.tw/~cjlin/libsvmtools/datasets/binary.html#a9a}{LIBSVM a9a} \citep{CC01a, Dua:2019} dataset. It is a binary classification dataset with $32{,}561$ points and $123$ features. For generating \cref{fig:a9a_grid_100_small} and \ref{fig:a9a_grid_100_reg_small} we use the entire dataset as a training set. For generating \cref{fig:a9a_grid_100_test_small} we split the dataset, using $20{,}000$ points as the training set and the rest as the validation set.

\textbf{Experiment 2.} As described in the body of the paper, each sub-plot (in \Cref{fig:a9a_grid_100_reg_small}) shows the performance of the different algorithms and their variants (as discussed above) on $\mu$-regularized logistic regression.  For the experiments we vary the values of $\mu$, $M$, and $R$, for fixed theoretical parallel runtime $KR$. Each of these is a different \emph{setting}, reflecting different computation-communication-accuracy trade-offs. For each of these settings, we run an algorithm configuration (with some $\beta$ and $\lambda=0$ for \fedac{}), with different learning rates $\eta$, and record the best loss (which is the Regularized ERM-loss on the entire dataset) obtained \emph{at any point} during optimization. Based on this run we pick the optimal learning rate for each setting for each algorithm.

Note that this means every point in a sub-plot is an individual experiment, for which we have tuned. Since we are only concerned with optimization, we tune and report the accuracy only on the training set (with $32{,}561$ points). Both the stochastic first-order oracle and the Hessian-vector product oracle, used by the respective algorithms, do sampling with replacement. Following \cite{yuan2020federated}, we ensure that our learning rate tuning grid was both wide and fine enough for the algorithms to achieve their optimal training losses. Since we are using stochastic algorithms which sample with replacement, without any definite order on the dataset (as opposed to making a single pass), the best suboptimality is potentially different for each run of the algorithm, even when choosing the same learning rate and initialization. Thus, once we have the optimal parameters ($\eta,\beta, \lambda$), we rerun the algorithm multiple times on the corresponding setting. The error bars represent the standard deviation of the best suboptimality obtained when using this optimal learning rate. 

Note that changing the regularization strength changes the optimal error value on the RERM task. Thus, to get the optimal loss value, we run exact Newton's method separately on each of these settings (for different values of $\mu$), until the algorithm had converged up to $\approx 12$ decimal places. We also ensure that our optimal loss values look similar to \cite{yuan2020federated}. The reported suboptimality in \Cref{fig:a9a_grid_100_small} is obtained after subtracting this optimal error value from the algorithms' error, followed by dividing this excess error with the optimal value.

\textbf{Experiment 1 (a).} As described in the body of the paper, each sub-plot in \Cref{fig:a9a_grid_100_small}  shows the performance of the different algorithms and their variants (as discussed above) on logistic regression for getting the best possible unregularized training loss. We train \fedac{} on an appropriately regularized RERM problem on the training set. All other variants minimize an ERM problem on the dataset. This is achieved through the internal regularization parameter for $\fedac{}$ as discussed above. We vary $M$ and $R$ while keeping the parallel runtime $KR$ to be fixed. Our hyperparameter tuning and repetitions are similar to the description for experiment 2 above, though we additionally tune the regularization strength $\lambda$ in the RERM problem for \fedac{}. All the algorithms do sampling with replacement on the training set, and can make multiple passes on the dataset. The optimal training loss is again obtained using 100 iterations of the Newton method.

\textbf{Experiment 1 (b).} As described in the body of the paper, each sub-plot in \Cref{fig:a9a_grid_100_test_small}  shows the performance of the different algorithms and their variants (as discussed above) on logistic regression for getting the best possible validation loss. We train \fedac{} on an appropriately regularized RERM problem on the training set. All other variants minimize an ERM problem on the training set. This is achieved through the internal regularization parameter for $\fedac{}$ as discussed above. We vary $M$ and $R$ while keeping the parallel runtime $KR$ to be fixed. Our hyperparameter tuning and repetitions are similar to the description for experiment 2 above, though we additionally tune the regularization strength $\lambda$ in the RERM problem for \fedac{}. We use the same split of training and validation datasets across the multiple repetitions. All the algorithms do sampling without replacement on the training set, and can make at most one pass on the dataset.   

\textbf{Hardware details.} All the experiments were performed on a personal computer, Dell XPS 7390. The total compute time (CPU) on the machine was about $200$ hours. No GPUs were used in any of the experiments in this paper. 

\subsection{Computational cost of a Hessian-vector product oracle}\label{sec:comp_cost}
Consider minimizing the function $F(x) = \mathbb{E}_z[f(x;z)]$ given access to various stochastic oracles. Note that \fedac{}, Minibatch SGD, and Local SGD are all first-order algorithms, in that they use a first-order stochastic oracle. Thus, each time they observe a stochastic gradient (as per Assumption \ref{assump:a1}), the oracle uses a single unit of randomness (e.g., a single data point, when thinking in terms of a finite training dataset).

In contrast, our main theoretical method \fedsn{} proceeds via two possible options: for Case 1 as established in \Cref{alg:quad-grad-access}, the algorithm queries both a stochastic gradient and a stochastic Hessian-vector product oracle at two independent samples, while in Case 2, we consider a different setting which allows the algorithm to observe, for any $x,u$, both a stochastic gradient and stochastic Hessian-vector product using the \emph{same} random sample $z$. We may note that the final theoretical guarantee for Case 1 differs by only a small constant factor from that of Case 2, and our results as presented in \Cref{thm:combined-alg} apply to both cases.

Thus, in order to maintain a fair comparison between these first-order methods and our practical method \fedsnlite{}, we have implemented \fedsnlite{} so that it also uses a single random sample (single data point), as outlined via Case 2 (\texttt{Same-Sample}) in \Cref{alg:quad-grad-access}.

In addition to keeping the number of samples consistent, it is important to understand how both oracle models compare computationally. Clearly every oracle call for \fedsnlite{} is at least as expensive as a first-order stochastic oracle, as it subsumes the latter. Moreover, it is unclear a priori if the Hessian-vector product can be computed efficiently (say, in as much time as vector addition or multiplication). However, it turns out that the Hessian-vector product for logistic regression can be efficiently computed, since the Hessian matrix for a given sample is actually rank one (i.e., an outer product of known vectors). This is generally true for loss functions which belong to the family of \textit{generalized linear models}. To see this, note that for loss functions of the form $F(x) = \sum_{i=1}^{N}\phi(b_ix^\top a_i)$, we have by a simple calculation that
\begin{equation*}
    \nabla_x F(x) = \sum_{i=1}^{N}\phi'(b_ix^\top a_i)b_i\cdot a_i\quad \textrm{and}\quad \nabla_x^2 F(x) := \nabla_x(\nabla_x F(x)) = \sum_{i=1}^{N}\phi''(b_ix^\top a_i)b_i^2\cdot a_ia_i^\top,
\end{equation*}
and so for any $v \in \R^d$, 
\begin{equation*}
    \nabla_x^2 F(x)v = \sum_{i=1}^{N}\phi''(b_ix^\top a_i)b_i^2\cdot a_i\cdot (a_i^\top v),
\end{equation*}
which means each summand can be calculated in $\mathcal{O}(d)$ time. When maximizing the log-likelihood in the logistic regression model with labels in $\{-1,1\}$, we need to minimize the following function,
\begin{equation*}
F(x) = \sum_{i=1}^{N}\left(-b_ix^\top a_i + \ln\left(1+\exp\left(b_ix^\top a_i\right)\right)\right),
\end{equation*}
which is an instance of the generalized linear models as considered above. Thus, in terms of vector operations, the Hessian-vector product oracle and the stochastic gradient oracle are \textit{asymptotically similar} in the logistic regression model. We also note that for a general class of differentiable functions, \cite{pearlmutter1994fast} provides a technique to compute the Hessian-vector product using two passes of backpropagation (in the context of neural networks). For instance, we may consider a twice-differentiable function $F: \R^d \mapsto \R$ and note that for a vector $v \in \R^d$,
\begin{align*}
\nabla^2_x F(x)v = \nabla_x\left(\nabla_xF(x)^\top v\right),    
\end{align*}
which can be obtained using two passes of backpropagation. 

For the scale of our problem, it turns out that the difference in the number of vector operations is outweighed by other implementational overhead (for, e.g., loops, memory operations, etc.). In \Cref{tab:wall_clock_times} we show the average per-step runtimes (over $250$ runs) for three different algorithms for $M=1, R=1, K=100000$.
\begin{table}[H]
\def\arraystretch{1.5}%
\centering
\begin{tabular}{ l c c} 
 \toprule
 \textbf{Algorithm} & \textbf{Avg. Runtime/Step} (in $10^{-5}$ sec.) & \textbf{Std. Dev. of Runtime/Step} (in $10^{-5}$ sec.)\\
 \midrule 
 \fedsnlite{} & 6.67 & 1.42\\ 
 Local SGD & 6.39 & 1.37\\
 \fedac{} & 7.48 & 1.58\\
 \toprule 
\end{tabular}
\caption{Comparing the wall-clock runtimes of different algorithms in our implementation. We run every repetition of each algorithm as an individual job, so that the runs are \textit{independent}, i.e., we don't introduce extraneous biases for any one algorithm (e.g., thermal throttling affecting the runtimes for a single algorithm). We run each algorithm $1700$ times so that the $95\%$ error margins for the average runtime/step estimate are within $1\%$ of its value. Specifically, we calculate the empirical standard deviation $\sigma$ of our average runtime/step estimate $\bar{x}$ by running it for $n=1700$ runs. Then we report the error margin $\delta = z\sigma/\sqrt{n}$, where $z=1.96$, is the z*-value from the standard normal distribution for a $95\%$ confidence level. We ensure that $n=1700$ is large enough, so that $\delta/\bar{x}\leq 0.01$. Finally we round $\bar{x}$ to two significant digits, respecting the $1\%$ error margin. We also report the standard deviation estimate of the runtime/step so that its $95\%$ error margins are within $4\%$ of its value.}
\label{tab:wall_clock_times}
\end{table}

\subsection{Additional experiments}\label{sec:add_exps}

In this section we present more comprehensive versions of the experiments presented in \Cref{sec:experiments}. 

\begin{figure}[H]
    \centering
    \includegraphics[width=\textwidth]{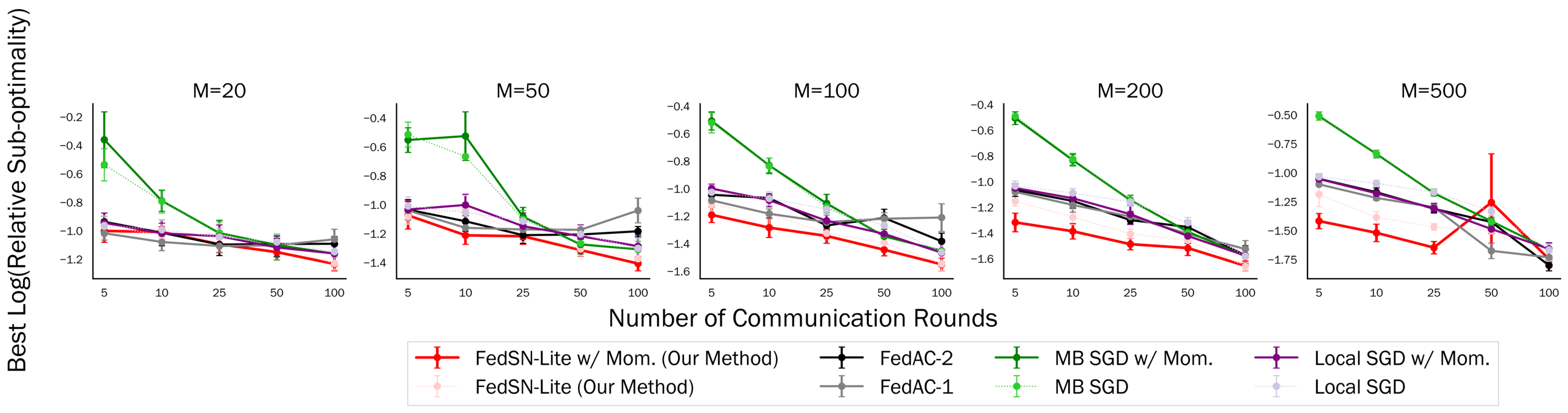}
    \caption{\small The same experiment as in \cref{fig:a9a_grid_100_small} but with a broader range of values of $M, \mu$. All optimization runs were repeated $70$ times.}
    \label{fig:a9a_grid_100}
\end{figure}

\begin{figure}[H]
    \centering
    \includegraphics[width=\textwidth]{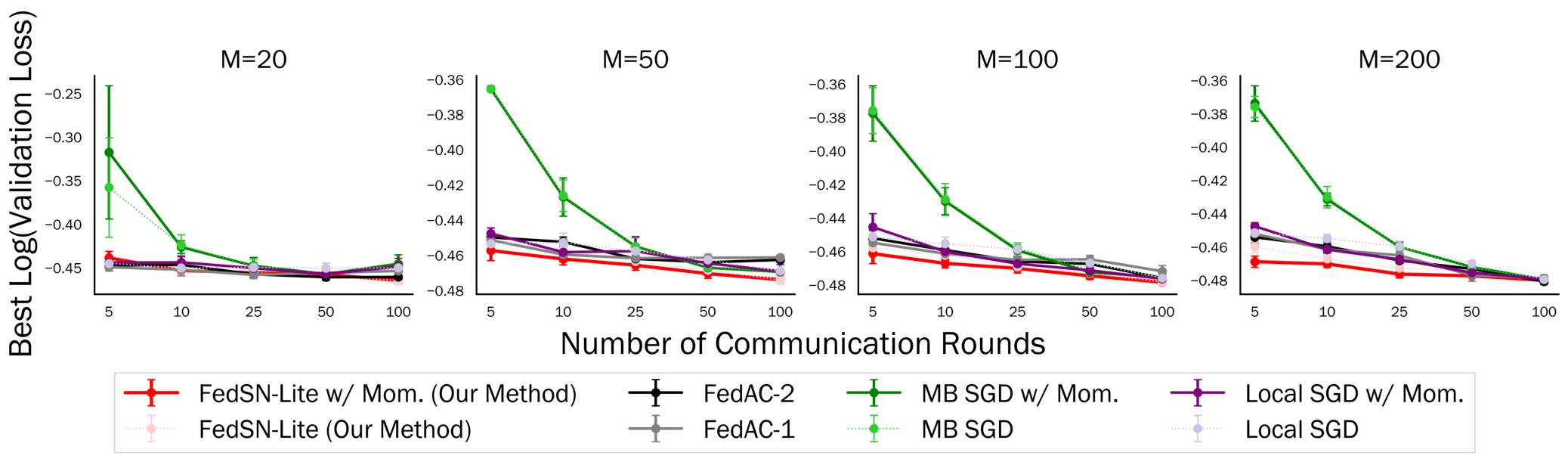}
    \caption{\small The same experiment as in \cref{fig:a9a_grid_100_test_small} but with a broader range of values of $M, \mu$. All optimization runs were repeated $20$ times.}
    \label{fig:a9a_grid_100_test}
\end{figure}

\begin{figure}[H]
    \centering
    \includegraphics[width=\textwidth]{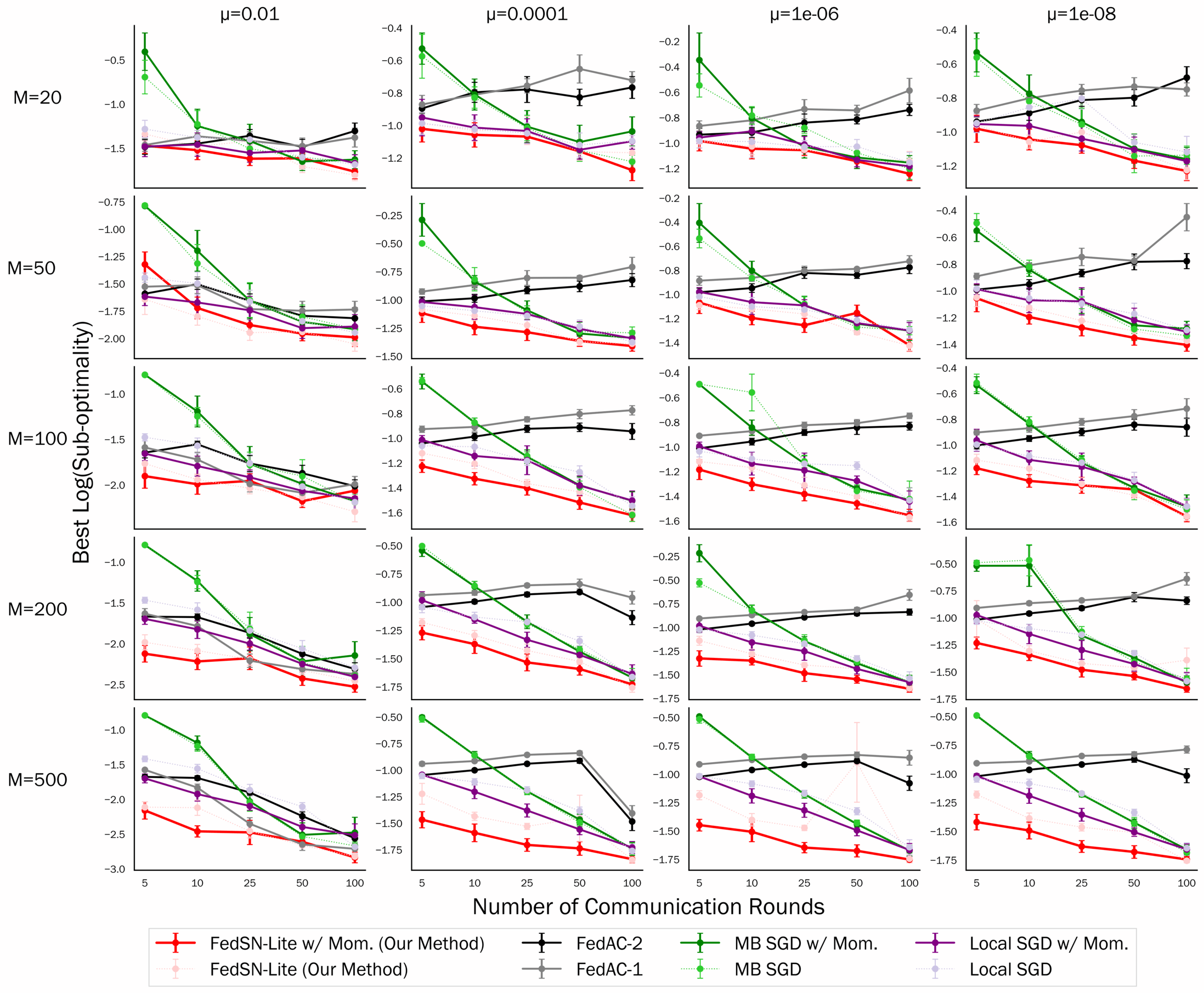}
    \caption{\small The same experiment as in \cref{fig:a9a_grid_100_reg_small} but with a broader range of values of $M, \mu$. All optimization runs were repeated $70$ times.}
    \label{fig:a9a_grid_100_reg}
\end{figure}

\begin{figure}[H]
    \centering
    \includegraphics[width=\textwidth]{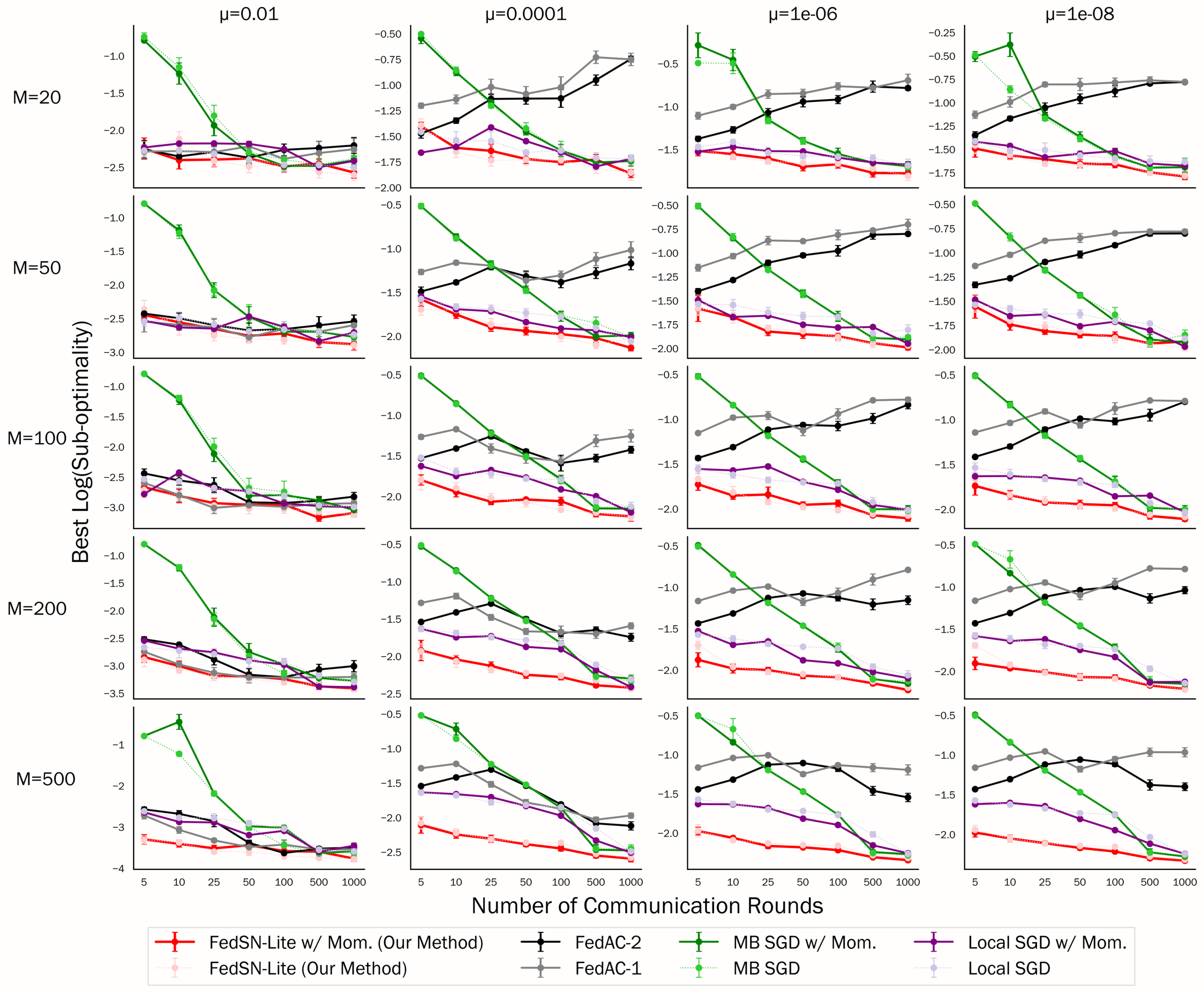}
    \caption{\small The same experiment as in \cref{fig:a9a_grid_100_reg_small} with $T=1000$ and a broader range of values of $M, \mu$. All optimization runs were repeated $20$ times.}
    \label{fig:a9a_grid_1000_reg}
\end{figure}

\end{document}